\documentclass[final,leqno,onefignum,onetabnum]{siamltex1213}

\usepackage{amsmath,amssymb}
\usepackage{graphicx}
\usepackage{xcolor}
\usepackage{algorithm}
\usepackage{algpseudocode}

\title{Tightness of a new and enhanced semidefinite relaxation for MIMO detection
\thanks{\textbf{Funding:} C. Lu's research was supported in part by NSFC grants 11701177 and 11771243 and Fundamental Research Funds for the Central Universities grant 2018ZD14. Y.-F. Liu's research was supported in part by NSFC grants 11671419, 11688101, and 11631013. W.-Q. Zhang's research was supported in part by NSFC grant U1836219. S. Zhang's research was supported in part by NSF grant CMMI-1462408.}
}

\author{Cheng Lu\footnotemark[2] \and Ya-Feng Liu\footnotemark[3] \and Wei-Qiang Zhang\footnotemark[4] \and Shuzhong Zhang\footnotemark[5]}

\newcommand{\T}{\mathsf{T}}

\begin{document}
\maketitle
\slugger{siopt}{xxxx}{xx}{x}{x--x}

\renewcommand{\thefootnote}{\fnsymbol{footnote}}
\footnotetext[2]{School of Economics and Management, North China Electric Power University, Beijing, China (lucheng1983@163.com).}
\footnotetext[3]{State Key Laboratory of Scientific and Engineering Computing, Institute of Computational
                Mathematics and Scientific/Engineering Computing, Academy of
                Mathematics and Systems Science, Chinese Academy of Sciences, Beijing
                100190, China (yafliu@lsec.cc.ac.cn).}
\footnotetext[4]{Department of Electronic Engineering, Tsinghua University, Beijing 100084, China (wqzhang@tsinghua.edu.cn).}
\footnotetext[5]{Department of Industrial and Systems Engineering, University of Minnesota, Minneapolis, MN 55455, USA (zhangs@umn.edu).}

\newtheorem{remark}{Remark}

\begin{abstract}
In this paper, we consider a fundamental problem in modern digital communications known as multiple-input multiple-output (MIMO) detection, which can be formulated as a complex quadratic programming problem subject to unit-modulus and discrete argument constraints. Various semidefinite relaxation (SDR) based algorithms have been proposed to solve the
problem in the literature. In this paper, we first show that the conventional SDR is {generally} not tight for the problem.
Then, we propose a new and enhanced SDR and show its tightness under an easily checkable condition, which essentially requires the level of the noise to be below a certain threshold.
The above results have answered an open question posed by So in \cite{So2010}. Numerical simulation results show that
our proposed SDR significantly outperforms the conventional SDR in terms of the relaxation gap.
\end{abstract}




\pagestyle{myheadings}
\thispagestyle{plain}
\markboth{C. LU, Y.-F. LIU, W.-Q. ZHANG AND S. ZHANG}{TIGHTNESS OF A NEW AND ENHANCED SDR FOR MIMO DETECTION}

\section{Introduction}
Multiple-input multiple-output (MIMO) detection is a fundamental problem in modern digital communications \cite{Yang}.
Mathematically, the input-output relationship of the MIMO channel can be modeled as \begin{equation}\label{rHv}r=Hx^{\ast}+v,\end{equation} where
\begin{itemize}
  \item [-] $r \in\mathbb{C}^{m}$ is
the vector of received signals,
\item [-] $H\in\mathbb{C}^{m\times n}$ is an
$m\times n$ complex channel matrix (for $n$ inputs and $m$ outputs with $m\geq n$),
\item [-] $x^{\ast}\in\mathbb{C}^{n}$ is the vector of transmitted
symbols, and
\item [-] $v\in\mathbb{C}^{m}$ is an additive {white circularly symmetric Gaussian noise}.
\end{itemize}
Assume that $M$-Phase-Shift Keying {($M$-PSK)} modulation scheme with $M\geq 2$ is adopted. Then each entry $x_i^{\ast}$ of $x^{\ast}$ belongs to a finite set of symbols, i.e.,
$$x_i^{\ast}\in \left\{e^{\textbf{i}\theta}\mid\theta=2j\pi/M,~j=0,1,\ldots,M-1\right\},~i=1,2,\ldots,n,$$ where $\textbf{i}$ is the imaginary unit.
The MIMO detection problem is to recover the vector of transmitted symbols $x^{\ast}$ from the vector of received signals $r$ based on the knowledge of the channel matrix $H$.
The mathematical formulation of the problem is 
\begin{equation*}\begin{array}{cl}
\displaystyle \min_{x \in \mathbb{C}^{n}} ~&~ \left\|Hx-r\right\|_2^2 \\[3pt] \tag{P}
\mbox{s.t.} ~&~ |x_i|^2=1,~\arg{(x_i)}\in \mathcal{A},~i=1,2,\ldots,n,
\end{array}\end{equation*}
where $\|\cdot\|_2$ denotes the Euclidean norm, $\arg{(\cdot)}$ denotes the argument of the complex number, and
\begin{equation}\label{setA}\mathcal{A}=\left\{0, 2\pi/M,\ldots,2(M-1)\pi/M\right\}.\end{equation}
It has been shown that minimizing the square error in the above MIMO detection problem (P) is equivalent to minimizing the probability of
{{vector}} detection error (see, e.g., \cite[Chapter 3]{Verdu}).

There are two related problems to {{the MIMO detection problem (P)}}. The first one is the following complex
quadratic programming problem:
\begin{equation}\begin{array}{cl}
\displaystyle \min_{x \in \mathbb{C}^{n}} ~&~ x^{\dag}Qx+2\textrm{Re}(c^{\dag}x) \\[3pt] \tag{CQP}
\mbox{s.t.} ~&~ |x_i|^2=1,~\arg(x_i) \in \mathcal{A},~i=1,2,\ldots,n,
\end{array}\end{equation}
where $Q\in \mathbb{C}^{n\times n}$ is a Hermitian matrix, $c\in\mathbb{C}^n$ is a complex column vector,
$(\cdot)^{\dag}$ denotes the conjugate transpose, and
$\textrm{Re}(\cdot)$ denotes the element-wise real part of a complex matrix/vector/number. Problem (P) is a special case of problem (CQP) where
$Q=H^{\dagger}H$ and $c=-H^{\dag} r.$
The second one is the following unit-modulus constrained quadratic programming problem:
\begin{equation}\begin{array}{cl}
\displaystyle \min_{x \in \mathbb{C}^{n}} ~&~  x^{\dag}Qx+2\textrm{Re}(c^{\dag}x)\\[3pt]
\mbox{s.t.} ~&~ |x_i|^2=1,~i=1,2,\ldots,n, \tag{UQP}
\end{array}\end{equation}
which can be seen as a continuous relaxation of problem (CQP) in the sense
that the discrete argument constraints $ \arg{(x_i)}\in \mathcal{A}$ for $i=1,\ldots,n$ in problem (CQP) are dropped.

Both problems (CQP) and (UQP) have been extensively studied due to their
broad applications. For instance, the Max-Cut problem \cite{Goemans1995} and the Max-3-Cut problem \cite{Goemans2004} are special cases of problem (CQP) with homogeneous
objective functions and with $M\in\{2,3\}$; the classical binary quadratic programming problem \cite{Beck} is also a special case of problem (CQP) with $M=2$. Moreover, problems (CQP) and (UQP) have found wide applications in signal processing and wireless communications, including MIMO detection \cite{Jalden2006,Jalden}, angular synchronization \cite{Bandeira2016,Singer},
phase retrieval \cite{Waldspurger}, and radar signal processing \cite{Maio2009,Wenqiang17,Soltanalian}. However, it is known that both problems (CQP) and (UQP) are NP-hard \cite{Zhang}, and the MIMO detection problem, as
a special case of problem (CQP), is also NP-hard \cite{Verdu1989}. Therefore, there is no polynomial-time algorithms which can solve these
problems to global optimality in general, unless P=NP.

{{
Back to the MIMO detection problem (P), various algorithms have been proposed to solve it. The sphere decoder algorithm \cite{Damen}, which can be seen as
a special branch-and-bound algrithm \cite{Murugan}, has been developed to find the exact solution of problem (P). However, the worst-case complexity
of the sphere decoder algorithm is exponential. To overcome the exponential complexity issue, some suboptimal algorithms have been proposed, including the
zero-forcing detector \cite{Kohno,Schneider}, the minimum mean-squared error detector \cite{Honig,Xie}, and the decision feedback detector \cite{Varanasi}.
However, the performance of these low-complexity suboptimal algorithms (in terms of the vector detection error rate) generally is poor. 

In the last two decades, the semidefinite relaxation (SDR) based algorithms have 
been widely studied in the signal processing and wireless communication community \cite{Luo2010}.
%
}}
For various SDR based algorithms for problems (CQP) and (UQP) under different signal processing and wireless communication scenarios, we refer the interested reader to  \cite{Bandeira2016,Jalden2006,Jalden,jacobsson2017quantized,Kisialiou,simo,Lu2017iccc,Ma2004,Maio2009,Wenqiang17,So2010,Tan,Waldspurger,Zhang} and the references therein. The SDR based algorithms
generally perform very well in some signal processing and wireless communication applications, as pointed out in \cite{Bandeira2014}. Similar observations have also
been made for MIMO detection \cite{Jalden2006,Jalden2008,Ma2004}, asynchronous multi-sensor
 data fusion \cite{Wenqiang17}, as well as angular synchronization \cite{Singer}. For the above applications, the SDR based algorithms proved to be impressively effective if the so-called signal-to-noise ratio (SNR) is high. {{In particular, for problem (P) with $M=2$, it has been shown in \cite{Jalden2006,Jalden2008} that
the SDR based algorithms can achieve the maximum possible diversity order (under the assumption that all entries of $H$ are independent and identically distributed (i.i.d.)
Gaussian variables).}}

Therefore, it has been a longstanding important question in the field as to understand why the performances of the SDR based algorithms are so remarkably good in practice. One line of research is directed to analyze the approximation ratios of the SDR based algorithms. Along this direction, the approximation ratios of some SDR based algorithms have been
analyzed in \cite{So2008,Zhang} for problems (CQP) and (UQP) with homogeneous positive semidefinite objective functions.
Another line of research is to identify conditions under which the SDRs are tight \cite{Bandeira2016,Jalden2006,Jalden,Kisialiou2005,Kisialiou,So2010}. The phenomenon that some nonconvex problems are equivalent to
their convex relaxations under certain conditions can also be regarded as a type of hidden convexity \cite{Sun}.
It is also worth remarking that, in some signal processing and wireless communication applications with high SNRs, even first-order algorithms are guaranteed to
converge to the global solution of these nonconvex problems \cite{Boumal,Liu,Liu2017}.


In this paper, we focus on the MIMO detection problem (P).
For ease of presentation, we define the \emph{tightness} of an SDR of the MIMO detection problem (P) as follows.
\begin{definition}
  An SDR of problem (P) is called tight if the following two conditions hold:
  \begin{itemize}
    \item [-] the gap between the SDR and problem (P) is zero; and
    \item [-] the SDR recovers the true vector of transmitted signals.
  \end{itemize}
\end{definition}
For the MIMO detection problem (P), the tightness
~of some SDRs has been studied in \cite{Jalden2006,Jalden,Kisialiou2005,Kisialiou,So2010}. 
In particular, So proved in \cite{So2010} that, for the case where $M=2$, there exists a tight SDR (see (BSDP) further ahead) if the inputs $H$ and $v$ in \eqref{rHv} satisfy \begin{equation}\label{M2condition}\lambda_{\min}\left(\textrm{Re} (H^{\dag}H)\right)>\|\textrm{Re}(H^{\dag}v)\|_{\infty},\end{equation} where
$\lambda_{\min}(\cdot)$ denotes the smallest eigenvalue of a given matrix and $\|\cdot\|_{\infty}$ denotes
the $L_\infty$-norm. 
In \cite{So2010}, So also posed the following open question:
Is the condition 
\begin{equation}\label{directcondition}\lambda_{\min}(H^{\dag} H)>\| \ H^{\dag} v\|_{\infty}\end{equation}sufficient for the conventional (complex) SDR being tight for problem (P) with $M\geq 3$?

The main contributions of our paper are twofold. First, we show that the conventional SDR is generally not tight for problem (P) and thus
answers an open question posed by So.
Second, we propose an enhanced SDR for problem (P), which is much tighter than the conventional SDR. We prove that our proposed enhanced SDR is tight for the case where $M\geq 3$ if the following condition is satisfied:
\begin{equation}\label{condition}
  \lambda_{\min}\left(H^{\dag}H\right)\sin\left(\frac{\pi}{M}\right)>\left\|H^{\dag}v\right\|_{\infty}.
\end{equation} To the best of our knowledge, {{for the case where $M\geq 3$}}, our new enhanced SDR is the first one to have a theoretical guarantee of tightness  if the SNR of the problem is sufficiently high (or equivalently the noise level of the problem is sufficiently low). Numerical results show that the new enhanced SDR performs significantly better than the conventional SDR in terms of the relaxation gap as well as the ability to recover the vector of transmitted signals.

%
%
%
%

The rest of this paper is organized as follows. In Section 2, we review the conventional SDR for problem (P) and show that it is generally not tight.
Then, we propose an enhanced SDR for problem (CQP) in Section 3 and prove it to be tight for problem (P) --- a special case of problem (CQP) --- if condition \eqref{condition} holds in Section 4. In Section 5, we present some numerical results to show the effectiveness of the newly proposed SDR. Finally, we conclude the paper in Section 6.


We adopt the following somewhat standard notations in this paper. For a given complex vector $x$, we use $x_i$ (or
$[x]_i$) to denote its $i$-th entry, $\|x\|_2$ to denote its Euclidean norm, {$\|x\|_1$ to denote its $L_1$-norm,} {$\|x\|_{\infty}$ to denote its $L_\infty$-norm,} and $\text{Diag}(x)$ to denote the diagonal matrix with the diagonal entries being $x.$
For a given complex Hermitian matrix $A$, $A\succeq 0$ means $A$ is positive semidefinite, $\textrm{Trace}(A)$ denotes the trace of $A$,
$A_{i,j}$ denotes the $(i, j)$-th entry of $A,$ and $A^{\dag}$ and $A^{\T}$ denotes the conjugate transpose and transpose of $A$, respectively. For two Hermitian matrices $A$ and $B$, $A\succeq B$ means $A-B\succeq 0$ and $A\bullet B$ means $\textrm{Re}\left(\textrm{Trace}(A^{\dag}B)\right).$ 
For any given matrix $C\in \mathbb{C}^{m\times n}$ (including the scalar case and the vector case),
we use $\textrm{Re}(C)$ and $\textrm{Im}(C)$ to denote the component-wise real and imaginary parts of $C$, respectively.
For a set $\mathcal{S}$, we use $\textrm{Conv}(\mathcal{S})$ to denote its convex envelope.
Finally, we use $\textbf{i}$ to denote the imaginary unit which
satisfies the equation $\textbf{i}^2=-1,$
{{and use $I_n$ to denote the $n$ by $n$ identity matrix}}.
In the remainder of this paper, we will focus on the MIMO detection problem.
We 
denote $$Q=H^{\dag} H,~c=-H^{\dag} r,~\text{and}~\lambda_{\min}=\lambda_{\min}(Q)$$ unless otherwise specified.


\section{Conventional semidefinite relaxations}
By introducing $X=xx^{\dag}$ and relaxing it to $X \succeq xx^{\dag}$ and dropping the argument constraints $\arg(x_i) \in \mathcal{A}$ for all $i=1,\ldots,n,$ we get the following conventional SDR for (CQP): 
\begin{equation}\begin{array}{cl}
\displaystyle \min_{x,\,X} ~&~  Q\bullet X +2 \textrm{Re}(c^{\dag}x) \\[3pt] \tag{CSDP}
\mbox{s.t.} ~&~ X_{i,i}=1,~i=1,2,\ldots,n,\\[3pt]
            ~&~ X \succeq xx^{\dag},
\end{array}\end{equation} where the variables $x \in \mathbb{C}^{n}$ and $X \in \mathbb{C}^{n\times n}.$
(CSDP) is a complex semidefinite program, which
has been widely used in the literature. Indeed, the approximation
algorithms in \cite{Bandeira2016,Jalden2006,Jalden,Kisialiou,Ma2004,Maio2009,So2008,So2010,Waldspurger,Zhang}
are all based on (CSDP).

In \cite{So2010}, So studied the tightness of an SDR for the MIMO detection problem.
For the special case where $M=2$, the argument constraints of (CQP) become the binary constraints $x\in\{-1,1\}^n$. Therefore, (CSDP) can be posed as the following real SDR:
\begin{equation}\begin{array}{cl}
\displaystyle \min_{x,\,X} ~&~  \textrm{Re}(Q) \bullet X +2 \textrm{Re}(c)^{\T}x \\[3pt] \tag{BSDP}
\mbox{s.t.} ~&~ X_{i,i}=1,~i=1,2,\ldots,n,\\[3pt]
            ~&~ X\succeq xx^{\T},
\end{array}\end{equation} where the variables $x \in \mathbb{R}^{n}$ and $X \in \mathbb{R}^{n\times n}.$
Based on (BSDP), So proved the following theorem.

\begin{theorem}[\cite{So2010}]\label{thm:so}
Suppose that $M=2.$ If the inputs $H$ and $v$ in \eqref{rHv} satisfy \eqref{M2condition}, then (BSDP) is tight for (P).
\end{theorem}

For the case where $M=2,$ Theorem 2.1 proposes a sufficient condition under which (BSDP) is tight for (P). In fact, for the same case, similar sufficient conditions for (BSDP) to be tight have been proposed in the literature. 
For instance, the condition proposed in \cite[Theorem 1]{Kisialiou2005} is  \begin{equation}\label{l1new}\lambda_{\min}\left(\textrm{Re}(H^{\dag} H)\right)>\left\| \textrm{Re} (H^{\dag} v)\right\|_{1}\end{equation}
{{and the one proposed in \cite{Kisialiou} is}} 
\begin{equation}\label{l2new}\lambda_{\min}\left(\textrm{Re}(H^{\dag} H)\right)>\left\| \textrm{Re} (H^{\dag} v)\right\|_{2}.\end{equation}
Since $\left\| \textrm{Re} (H^{\dag} v)\right\|_{p}\geq \left\| \textrm{Re} (H^{\dag} v)\right\|_{\infty}$ for $p\in\{1,\,2\}$, condition \eqref{M2condition}
is weaker than the above two.  

{{
In addition to the above sufficient conditions, the following sufficient and necessary condition
for the case where $M=2$ is also proposed in \cite[Theorem 1]{Jalden} and \cite[Theorem 7.1]{Jalden2006}:
Let $H\in\mathbb{R}^{m\times n}$, $v\in \mathbb{R}^m,$ and $x^{\ast}\in\{-1,1\}^n$,  (BSDP) is tight if and only if
$$v\in\left\{u\in\mathbb{R}^m|~ H^{\T}H+[\text{Diag}(x^{\ast})]^{-1} \text{Diag}(H^{\T} u)\succeq 0\right\}.$$
The above condition can be extended to the complex case \cite{Jalden} where
$H\in\mathbb{C}^{m\times n}$, $v\in \mathbb{C}^m,$ and $x^{\ast}\in\{-1,1\}^n$ as follows:
$$v\in\left\{u\in\mathbb{C}^m|~ \textrm{Re}(H^{\dag}H)+[\text{Diag}(x^{\ast})]^{-1} \text{Diag}[\textrm{Re}(H^{\dag} u)]\succeq 0\right\}.$$
It is simple to verify that condition \eqref{M2condition}, as well as conditions \eqref{l1new} and \eqref{l2new}, can be derived from the above sufficient and necessary condition.
}}

For more general cases where $M\geq 3$,
So posed an open question in \cite{So2010}: Is condition \eqref{directcondition} sufficient for (CSDP) to be tight for (P)?
Next, we show that the answer to this question is negative. More specifically, we show that (CSDP) is not tight for almost all instances of (P) with nonzero random noise
in the sense that the probability that there is no gap between (CSDP) and (P) is zero.

To tackle the problem, we first derive a necessary condition for (CSDP) to be tight.
Denote
\begin{align}\label{Qcy}
\hat{Q}=\begin{bmatrix}
\textrm{Re}(Q) ~&-\textrm{Im}(Q)\\[3pt]
\textrm{Im}(Q) ~&\textrm{Re}(Q)\\
\end{bmatrix},~ \hat{c}=\begin{bmatrix} 
\textrm{Re}(c)\\[3pt]
\textrm{Im}(c) \\
\end{bmatrix},~\text{and}~y=\begin{bmatrix} 
\textrm{Re}(x) \\[3pt]
\textrm{Im}(x)\\
\end{bmatrix}.
\end{align}
Then (UQP) in its real form can be written as 
\begin{equation}\label{RUQP}\begin{array}{cl}
\displaystyle \min_{y\in \mathbb{R}^{2n}} ~&~   y^{\T}\hat{Q} y+2\hat{c}^{\T}y\\[3pt]
\mbox{s.t.} ~&~ y_i^2+y_{n+i}^2=1,~i=1,2,\ldots,n.
\end{array}\end{equation} 
The Lagrangian function of \eqref{RUQP} is $$L(y;\lambda)= y^{\T} \hat{Q} y+2\hat{c}^{\T}y+
\sum_{i=1}^n\lambda_i\left(y_{i}^2+y_{n+i}^2-1\right),$$ where $\lambda_i$ is the Lagrange multiplier corresponding to the constraint $y_i^2+y_{n+i}^2=1.$
{{The Lagrangian dual problem of \eqref{RUQP} is
\begin{equation}\label{LD}
\max_{\lambda\in \mathbb{R}^n}~d(\lambda) \tag{D}
\end{equation} where $d(\lambda):=\min_{y\in \mathbb{R}^{2n}} L(y;\lambda)$ is the dual function of problem \eqref{RUQP}.
It is well-known that problem (D) is equivalent to (CSDP) \cite{Lemarechal,Shor}, i.e., the optimal value of problem (D) is equal to that of problem (CSDP).
%

Next, we derive the KKT optimality condition for problem \eqref{RUQP}. It is simple to verify that the linear independent constraint qualification (LICQ) holds for problem \eqref{RUQP}. Let $x^{\ast}$ be an optimal solution of (UQP) and $y^{\ast}$ be the corresponding optimal
solution of problem \eqref{RUQP}. Then there exist $\left\{{\lambda}_i^{\ast}\in \mathbb{R}\right\}_{i=1}^n$ such that
%
%
}}
%
$$\left. \frac{\partial L(y; \lambda)}{\partial y_i}\right|_{y=y^{\ast}}=2\left[\hat{Q} y^{\ast}\right]_{i}+2\hat{c}_i+2{\lambda}_{i} {y^{\ast}_i}=0,~i=1,2,\ldots,n$$
and $$\left. \frac{\partial L(y; \lambda)}{\partial y_{i+n}}\right|_{y=y^{\ast}}=2\left[\hat{Q} y^{\ast}\right]_{i+n}+2\hat{c}_{i+n}+2{\lambda}_{i} {y^{\ast}_{i+n}}=0,~i=1,2,\ldots,n.$$
Since $$2\left[\hat{Q} y^{\ast}\right]_{i}+2\hat{c}_i+2{\lambda}_{i} y^{\ast}_i= 2\textrm{Re}\left(\left[Q{x^{\ast}}\right]_i+c_i\right) +2{\lambda}_{i}\textrm{Re}\left({x^{\ast}_i}\right),~i=1,2,\ldots,n$$
and $$2\left[\hat{Q} y^{\ast}\right]_{i+n}+2\hat{c}_{i+n}+2{\lambda}_{i} {y^{\ast}_{i+n}}=2\textrm{Im}\left(\left[Q{x^{\ast}}\right]_i+c_i\right) +2{\lambda}_{i}\textrm{Im}\left({x^{\ast}_i}\right),~i=1,2,\ldots,n,$$
the KKT condition of (UQP) in the complex form becomes $$\left[Qx^{\ast}\right]_i+c_i+{\lambda}_{i}{x^{\ast}_i}=0,~i=1,2,\ldots,n.$$
Now, we prove the next theorem.

\begin{theorem}
Suppose that $M\geq 2$. If (CSDP) is tight for (P), then there exist $\left\{\lambda_i^{\ast}\in \mathbb{R}\right\}_{i=1}^n$
such that \begin{equation}\label{optcondcqp}\left[H^{\dag} v\right]_i=\lambda_i^{\ast} x_i^{\ast},~i=1,2,\ldots,n,\end{equation} where
$H,~x^{\ast},~\text{and}~v$ are given in \eqref{rHv}. 
\end{theorem}
\begin{proof}
Let $\lambda^{\ast}\in \mathbb{R}^n$ be the optimal solution of problem (D). If (CSDP) is tight for (P),
then the KKT condition of (UQP) is satisfied {{at $x^{\ast}$}}, i.e.,
$$\left[Qx^{\ast}+c\right]_i+\lambda_i^{\ast} x_i^{\ast}=0,~i=1,2,\ldots,n.$$ This, together with
$Qx^{\ast}+c=H^{\dagger} Hx^{\ast}-H^{\dag} r=-H^{\dag}v,$ immediately implies the desired result \eqref{optcondcqp}.
\end{proof}

{{ The conditions in \eqref{optcondcqp} imply that either $[H^{\dagger}v]_i$ and $x_i^{\ast}$ have the same phase or $[H^{\dagger}v]_i=0$ for all $i=1,2,\ldots,n$.
However, this is generally not true in real applications. In particular, the noise vector $v$ is often assumed to follow a circularly symmetric complex
Gaussian distribution, and the probability that the event
$$\left\{ [H^{\dagger}v]_i=0 \right\}\bigcup \left\{\arg \left( [H^{\dagger}v]_i \right)=\arg \left(x_i^{\ast} \right)\right\}$$ happens
is zero for each $i=1,\ldots,n$.
Therefore, the probability that all conditions in \eqref{optcondcqp} are simultaneously satisfied is zero.}}
This immediately implies that (CSDP) is generally not tight for (P) (regardless of the condition in \eqref{directcondition}). This answers the open question posed in \cite{So2010}. It is also {{worth noting that (CSDP) and (BSDP) are not equivalent to each other even for the case where $M=2$}}.




\section{An enhanced semidefinite relaxation}

%
%
%
In this section, we propose an enhanced SDR for (CQP).
Recall that the (discrete) argument constraints $\arg\left(x_i\right)\in \mathcal{A}$ for $i=1,2,\ldots,n$ in (CQP) are
ignored in its conventional relaxation (CSDP). The idea of designing the enhanced SDR for (CQP) is to better exploit the structure of
the argument constraints and develop valid linear constraints for them to tighten (CSDP). Since these valid linear
constraints are based on the real form of (CSDP), we first reformulate
(CSDP) as the following real SDR: 
\begin{equation}\begin{array}{cl}
\displaystyle \min_{y,\,Y} ~&~ \hat{Q} \bullet Y +2\hat{c}^{\T}y\\[3pt]
\mbox{s.t.} ~&~ Y_{i,i}+Y_{n+i,n+i}=1,~i=1,2,\ldots,n, \\[5pt]
~&~ \begin{bmatrix}
1 &y^{\T}\\
y &Y\\
\end{bmatrix}
\succeq 0,\tag{RSDP}
\end{array}\end{equation}
where the variables $y\in \mathbb{R}^{2n}$ and $Y\in \mathbb{R}^{2n\times 2n}$ and
$\hat{Q}$, $\hat{c},$ and $y$ are defined in \eqref{Qcy}.

Remark that our real reformulation (RSDP) of complex (CSDP) is not the same as the ones in \cite{Goemans2004} and \cite{Waldspurger}.
The dimension of the matrix variable in (RSDP) is $2n+1$ while the one of the matrix variable in \cite{Goemans2004} and \cite{Waldspurger} is $2n+2$. Hence, the equivalence between (CSDP) and (RSDP) cannot be shown by using the same argument in \cite{Goemans2004} and \cite{Waldspurger}. An equivalence proof of (CSDP) and (RSDP) is provided in Appendix A.

%

Next, we develop some valid linear constraints for the argument constraints based on (RSDP), which leads to an enhanced SDR for (CQP).
First, for (RSDP), we define the following $3\times 3$ matrices:
\begin{equation}\label{Yi}\mathcal{Y}(i):=\begin{bmatrix}
1 ~&y_{i} ~&y_{n+i}\\[3pt]
y_{i} ~&Y_{i,i} ~&Y_{i,n+i}\\[3pt]
y_{n+i} ~&Y_{n+i,i} ~&Y_{n+i,n+i}\\
\end{bmatrix},~i=1,2,\ldots,n.\end{equation} 
{{For each $i=1,\ldots,n$,}} $\mathcal{Y}(i)$ contains the following $5$ variables in (RSDP) (due to its symmetry):
$$y_i,~y_{n+i},~Y_{i,i},~Y_{i,n+i},~\text{and}~Y_{n+i,n+i}.$$
{{
From the definition of $y$ in \eqref{Qcy}, we have
$y_i=\textrm{Re}(x_i)$ and $y_{n+i}=\textrm{Im}(x_i)$. Since
$|x_i|^2=1$ and $\arg{(x_i)}\in \mathcal{A}$}}, it follows  $$\left(y_i,\,y_{n+i}\right)\in \left\{(\cos\left(\theta\right),\,\sin\left(\theta\right))\mid\theta={2j\pi}/{M},~j=0,1,\ldots,M-1 \right\}.$$
Define the following $3\times 3$ real symmetric matrices:
\begin{equation}\label{Pk}P_j=  \begin{bmatrix}
1 \\[3pt]
\cos\left(\frac{2j\pi}{M}\right)\\[5pt]
\sin\left(\frac{2j\pi}{M}\right)
\end{bmatrix}
\begin{bmatrix}
1 &\cos\left(\frac{2j\pi}{M}\right) &\sin\left(\frac{2j\pi}{M}\right)
\end{bmatrix},~j=0,1,\ldots,M-1.\end{equation}
Then, each of $\mathcal{Y}(i)$ must equal one of matrices
$P_j$ with $j=0,1,\ldots,M-1,$ i.e.,
$$\mathcal{Y}(i)\in\left\{P_0, ~P_1,\ldots,P_{M-1}\right\},~i=1,2,\ldots,n.$$
{{The convex envelope}} of the above constraints are
$$\mathcal{Y}(i)\in \textrm{Conv} \left\{P_0, ~P_1,\ldots,P_{M-1}\right\},~i=1,2,\ldots,n,$$
which are equivalent to
$$\mathcal{Y}(i)=\sum_{j=0}^{M-1} t_{i,j} P_j,~\sum_{j=0}^{M-1} t_{i,j}=1,~t_{i,j}\geq 0,~j=0,1,\ldots,M-1,~i=1,2,\ldots,n.$$
Due to the symmetry of $\mathcal{Y}(i),$ the constraint $\mathcal{Y}(i)=\sum_{j=0}^{M-1} t_{i,j} P_j$ can be explicitly expressed as the following $5$ linear constraints:
\begin{equation}\label{linearcons}
\begin{array}{rl}
\displaystyle y_i&\!\!\!\!=\displaystyle\sum_{j=0}^{M-1} t_{i,j} \cos\left(\frac{2j\pi}{M}\right),~y_{n+i}=\sum_{j=0}^{M-1} t_{i,j} \sin\left(\frac{2j\pi}{M}\right), \\[12pt]
\displaystyle Y_{i,i}&\!\!\!\!=\displaystyle\sum_{j=0}^{M-1} t_{i,j} \cos^2\left(\frac{2j\pi}{M}\right),~Y_{n+i,n+i}=\sum_{j=0}^{M-1} t_{i,j} \sin^2 \left(\frac{2j\pi}{M}\right), \\[12pt]
\displaystyle Y_{i,n+i}&\!\!\!\!=\displaystyle\sum_{j=0}^{M-1} t_{i,j} \cos \left(\frac{2j\pi}{M}\right) \sin\left(\frac{2j\pi}{M}\right).
\end{array}
\end{equation}Obviously, the above equations 
and $\sum_{j=0}^{M-1} t_{i,j}=1$ imply $Y_{i,i}+Y_{n+i,n+i}=1.$ By dropping redundant constraints, we get the following enhanced SDR for (CQP):
\begin{equation}
\begin{array}{cl}
\displaystyle \min_{y,\,Y,\,t} ~&~  \hat{Q} \bullet Y + 2\hat{c}^{\T}y\\[3pt]
\mbox{s.t.} ~&~\displaystyle \mathcal{Y}(i)=\sum_{j=0}^{M-1} t_{i,j} P_j,~\sum_{j=0}^{M-1} t_{i,j}=1,~i=1,2,\ldots,n,\\[12pt] \tag{ERSDP}
~&~t_{i,j}\geq 0,~i=1,2,\ldots,n,~j=0,1,\ldots,M-1,\\[6pt]
~&~\left[
   \begin{array}{cc}
     1 & y^{\T} \\
     y & Y \\
   \end{array}
 \right]\succeq 0,
\end{array}
\end{equation} where the variables $y\in \mathbb{R}^{2n},$ $Y\in \mathbb{R}^{2n\times 2n},$ $t\in \mathbb{R}^{n\times M},$ $\hat Q$ and $\hat c$ are defined in \eqref{Qcy}, $\mathcal{Y}(i)$ is defined in \eqref{Yi}, and $P_j$ is defined in \eqref{Pk}.

We term the above real SDP ``ERSDP'', since $\mathcal{Y}(i)$ for all $i=1,2,\ldots,n$ in (ERSDP) are constrained in the convex envelope with the \textbf{E}xtreme points being $P_0, P_1,\ldots, P_{M-1},$ i.e., the variables $y_i,~y_{n+i},~Y_{i,i},~Y_{i,n+i},~\text{and}~Y_{n+i,n+i}$ must satisfy the linear constraints given in \eqref{linearcons}, which is the main difference between (ERSDP) and (RSDP). Hence, (ERSDP) is (strictly) tighter than (RSDP) (which is equivalent to (CSDP)).

{{
It is worth noting that the proposed (ERSDP) is not the first SDR for the MIMO detection problem (P) that is stronger/tighter than (CSDP).
For instance, both the SDRs proposed in \cite{Fan,Mobasher}
are customized for problem (P). The proposed (ERSDP)
is different from those in \cite{Fan,Mobasher} in the sense that the matrix
variables in these relaxations are lifted from different spaces. 
}}

\section{Tightness of (ERSDP)}\label{sec:exactness}
In this section, we study the tightness of the newly proposed SDP relaxation (ERSDP).

Let us first look at a case where $M=2.$ In this case, $$P_0 = \left[
          \begin{array}{ccc}
       ~1 & ~1 & ~0\\
~1 & ~1 & ~0\\
~0 & ~0 & ~0\\
          \end{array}
        \right]
~\text{and}~P_1 = \left[
          \begin{array}{ccc}
~1 & -1 & ~0\\
-1 & ~1 & ~0\\
~0 & ~0 & ~0\\
          \end{array}
        \right]$$ and the linear constraints in \eqref{linearcons} reduce to
        $$Y_{i,i}=1,~y_i=t_{i,1}-t_{i,2},~\text{and}~y_{n+i}=Y_{n+i,i}=Y_{n+i,n+i}=0.$$ 
By dropping the zero blocks in the matrix $Y$, (ERSDP) reduces to (BSDP). Therefore, it follows from Theorem \ref{thm:so} that condition \eqref{M2condition} is sufficient for
(ERSDP) to be tight for problem (P) where $M=2$.

In the remainder of this section, we study the tightness of (ERSDP) for (P) where $M\geq 3$. We prove that condition \eqref{condition} is sufficient for (ERSDP) to be tight for (P).
Our proof consists of two main steps: \textbf{Step I}, we derive a complex SDR called (CSDP2) based on (ERSDP) and show that
(ERSDP) is tighter than (CSDP2); \textbf{Step II}, we show that (CSDP2) is tight for (P) where $M\geq 3$ under certain condition and hence
(ERSDP) is also tight for (P) under the same condition.

\textbf{Step I.} We derive a complex SDR to be called (CSDP2) from (ERSDP). Recall that, for any feasible solution $(y,\,Y)$ of (ERSDP), $\mathcal{Y}(i)$ in \eqref{Yi}
must lie in the convex envelope of $\{P_0,\ldots,P_{M-1}\},$ which implies that
\begin{equation}\label{yiyni}\left(y_i,\,y_{i+n}\right)\in \textrm{Conv}\left\{\left(\cos\left(\theta\right),\,\sin\left(\theta\right)\right)\mid\theta\in \mathcal{A}\right\},\end{equation} 
where $\mathcal{A}$ is defined in \eqref{setA}. Note that $y_i$ and $y_{i+n}$ correspond to the real and imaginary parts of the complex variable $x_i$. Then, one can show that \eqref{yiyni} is equivalent to $x_i\in \mathcal{F},$
where $\mathcal{F}$ is the convex envelope of the set 
$$\mathcal{D}=\left\{ w\mid|w|=1,~\arg(w)\in \mathcal{A}\right\}.$$ For the case where $M\geq 3$, by using a similar argument of showing Proposition 1 in \cite{luliu2017tsp}, the set $\mathcal{F}$
can be represented as
 $$\left\{ w\in \mathbb{C}\mid \textrm{Re}(a^{\dag}_j w) \leq  \cos\left(\frac{\pi}{M}\right),~j=0,1,\ldots,M-1\right\},$$ where
\begin{equation}\label{atheta}a_j=e^{ \textbf{i}\theta_j}~\text{and}~\theta_j=\frac{(2j-1)\pi}{M}.\end{equation}
Furthermore, $\textrm{Re}(a^{\dag}_j w) \leq  \cos\left(\frac{\pi}{M}\right)$ is equivalent to
$$\cos\left(\theta_j\right)\textrm{Re}(w)+\sin\left(\theta_j\right)\textrm{Im}(w) \leq  \cos\left(\frac{\pi}{M}\right).$$ The line associated with the above half plane connects the two points $$\cos\left(\frac{2(j-1)\pi}{M}\right)+\sin\left(\frac{2(j-1)\pi}{M}\right)\textbf{i}~\text{and}~\cos\left(\frac{2j\pi}{M}\right)+\sin\left(\frac{2j\pi}{M}\right)\textbf{i}$$ in the complex domain. See Figure 1 for an illustration of the set $\mathcal{F}$ for the case where $M=8.$

Based on the above observations, we obtain the following complex SDR:
\begin{equation}
\begin{array}{cl}
\displaystyle \min_{x,\,X} ~&~ Q\bullet X + 2\textrm{Re}(c^{\dag}x)  \\[3pt]
\mbox{s.t.} ~&~ X_{ii}=1,~i=1,2,\ldots,n,\tag{CSDP2}\\[3pt]
    ~&~ \textrm{Re}(a^{\dag}_j x_i) \leq  \cos\left(\frac{\pi}{M}\right),~i=1,2,\ldots,n,~j=0,1,\ldots,M-1,\\[5pt]
            ~&~ \begin{bmatrix}
1 &x^{\dag}\\
x &X\\
\end{bmatrix}
\succeq 0,
\end{array}
\end{equation} where the variables $x \in \mathbb{C}^{n}$ and $X \in \mathbb{C}^{n\times n}.$

Two remarks on the comparison between (CSDP2) and  (ERSDP) are in order. First,
since the constraints
$\textrm{Re}(a^{\dag}_j x_i) \leq  \cos\left(\frac{\pi}{M}\right)$ for all $i=1,2,\ldots,n$ and $j=0,1,\ldots,M-1$ in (CSDP2)
are derived from (ERSDP), we know that (ERSDP) is at least as tight as (CSDP2). In fact, we have the following result.
\begin{theorem}\label{thmtighter}
  (ERSDP) is tighter than (CSDP2).
\end{theorem}

A rigorous proof of Theorem \ref{thmtighter} can be found in Appendix B.
Therefore, if (CSDP2) is tight for (P), then (ERSDP) must also be tight.
Second, (CSDP2) is tight enough for us to derive our main results in Theorems \ref{thmgeneral} and \ref{thmM3exact}. 
Our analysis is based on a simpler reformulation of (CSDP2) and its dual.

\begin{figure}[!t]
\centering
\includegraphics[width=10.5cm]{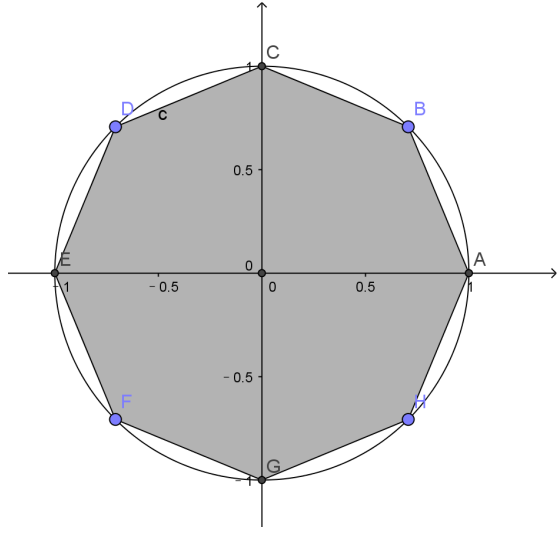}
\caption{An illustration of the set $\mathcal{F}$ for the case where $M=8$, which can be represented by $8$ linear constraints.}
\end{figure}

%

\textbf{Step II.} We show the tightness of (CSDP2) in this part. 
Our {{proof is similar to the one of showing Theorem 1 in \cite{Jalden}}. The basic idea
is to construct a dual feasible solution such that $\left(x^{\ast},\,X^{\ast}\right)$ and the constructed
dual solution jointly satisfy the KKT optimality conditions of (CSDP2) under some conditions, where $X^{\ast}=x^{\ast}\left(x^{\ast}\right)^{\dag}$.} The dual problem of (CSDP2) is 
\begin{equation}\label{dualproblem}\begin{array}{cl}
\displaystyle \max_{\lambda,\,\tau,\,\mu} ~&~ \displaystyle  \tau-\sum_{i=1}^n \lambda_i - \cos\left(\frac{\pi}{M}\right) \sum_{i=1}^n \sum_{j=0}^{M-1}\mu_{i,j} \\[8pt]
\mbox{s.t.} ~&\begin{bmatrix}
-\tau &\left(c+g\right)^{\dag}\\[3pt]
c+g ~&Q+ \text{Diag}(\lambda) \\
\end{bmatrix}\succeq 0,\\[13pt]
&\mu_{i,j}\geq 0,~i=1,2,\ldots,n,~j=0,1,\ldots,M-1,
\end{array}\end{equation}
where the variables $\lambda\in \mathbb{R}^n,\,\tau\in \mathbb{R},\,\mu\in \mathbb{R}^{n\times M}$ and
$g=[g_1, g_2,\ldots,g_n]^{\T}$ with \begin{equation}\label{gi}g_i= \sum_{j=0}^{M-1}\frac{\mu_{i,j}}{2}a_j ,~i=1,2,\ldots,n.\end{equation}
For completeness, a detailed derivation of the above dual problem is provided in Appendix C.
Since both problems (CSDP2) and its dual \eqref{dualproblem} are strictly feasible, 
it follows that the primal-dual Slater's conditions are satisfied.
Suppose that $(x,\,X)$ is an optimal solution of (CSDP2).
Then there {{must exist a feasible solution}} $\left({\lambda},\,{\tau},\,{\mu}\right)$ (to problem \eqref{dualproblem}) such that $(x,\,X)$ and $\left({\lambda},\,{\tau},\,{\mu}\right)$ jointly satisfy the following complementarity conditions:
\begin{equation}\label{complement1}\mu_{i,j}\left(\cos\left(\frac{\pi}{M}\right)-\textrm{Re}\left(a^{\dag}_j x_i\right)\right)=0,~i=1,2,\ldots,n,~j=0,1,\ldots,M-1\end{equation}
and
\begin{equation}\label{complement2}
\begin{bmatrix}
-{\tau} ~&\left(c+{g}\right)^{\dag}\\[3pt]
c+{g} ~&Q+ \text{Diag}(\lambda) \\
\end{bmatrix}
\bullet
\begin{bmatrix}
1 ~& x^{\dag}\\[2pt]
x ~& X \\
\end{bmatrix}=0.
\end{equation}

%
%
We are ready to present our main results in this paper.

{{

\begin{theorem}\label{thmgeneral}
Suppose that $M\geq 3.$ Let $x^{\ast}$ with \begin{equation}\label{xisi}x_i^{\ast}=e^{2\pi s_i \textbf{i}/M},~s_i\in\{0,1,\ldots,M-1\},~i=1,2,\ldots,n\end{equation} be the vector of transmitted signals.
Define \begin{equation}\label{ti}t_i=\left\{\begin{array}{cl}
s_i+1,&\textrm{if}~s_i<M-1;\\[3pt]
0,&\textrm{if}~s_i=M-1,
\end{array}\right.~i=1,2,\ldots,n.\end{equation}
Then $\left(x^{\ast},\,X^{\ast}\right)$ is an optimal solution to
(CSDP2) if and only if there exist \begin{equation}\label{existence}\bar{\lambda}_i\in \mathbb{R},~\bar{\mu}_{i,s_i}\geq 0,~\text{and}~\bar{\mu}_{i,t_i}\geq 0,~i=1,2,\ldots,n\end{equation}
such that $H$ and $v$ in \eqref{rHv} satisfy \begin{equation}\label{Hv}[H^{\dag}v]_i=\bar{\lambda}_i x_i^{\ast}+ \frac{\bar{\mu}_{i,s_i}}{2} a_{s_i}+ \frac{\bar{\mu}_{i,t_i}}{2} a_{t_i},~i=1,2,\ldots,n,\end{equation}
and \begin{equation}\label{Qlambdabar}Q+\textrm{Diag}(\bar{\lambda})\succeq 0.\end{equation} Furthermore, if $Q+\textrm{Diag}(\bar{\lambda})\succ 0$, then
$\left(x^{\ast},\,X^{\ast}\right)$ is the unique optimal solution to (CSDP2).
\end{theorem}
}}
\begin{proof}
{{
We first prove the necessary direction. Assume that {{$\left(x^{\ast},\,X^{\ast}\right)$}} is an optimal solution of (CSDP2) and
$\left({\lambda}^{\ast},\,{\tau}^{\ast},\,{\mu}^{\ast}\right)$ is an optimal solution of dual problem \eqref{dualproblem}. Denote ${\Lambda}^{\ast}=\textrm{Diag}\left({\lambda}^{\ast}\right)$. Then, from the optimality conditions, we have
\begin{equation}\label{dualfeasible}\begin{bmatrix}
-{\tau}^{\ast} ~&\left(c+{g}^{\ast}\right)^{\dag}\\[3pt]
c+{g}^{\ast} ~&Q+ {\Lambda}^{\ast} \\
\end{bmatrix}\succeq 0,\end{equation}
\begin{equation}\label{complementx2}
\begin{bmatrix}
-{\tau}^{\ast} ~&\left(c+{g}^{\ast}\right)^{\dag}\\[3pt]
c+{g}^{\ast} ~&Q+ {\Lambda}^{\ast} \\
\end{bmatrix}
\bullet
\begin{bmatrix}
1 ~& \left(x^{\ast}\right)^{\dag}\\[2pt]
x^{\ast} ~& X^{\ast} \\
\end{bmatrix}
=
\begin{bmatrix}
1 \\[2pt]
x^{\ast}\\
\end{bmatrix}^{\dag}
\begin{bmatrix}
-{\tau}^{\ast} ~&\left(c+{g}^{\ast}\right)^{\dag}\\[3pt]
c+{g}^{\ast} ~&Q+ {\Lambda}^{\ast} \\
\end{bmatrix}
\begin{bmatrix}
1 \\[2pt]
x^{\ast} \\
\end{bmatrix}
=0,
\end{equation} and
\begin{equation}\label{mucomplementarity}\mu^{\ast}_{i,j}\left(\cos\left(\frac{\pi}{M}\right)-\textrm{Re}\left(a^{\dag}_j x^{\ast}_i\right)\right)=0,~i=1,2,\ldots,n,~j=0,1,\ldots,M-1.\end{equation}
From \eqref{dualfeasible}, we immediately get \begin{equation}\label{Qlambda}Q+ {\Lambda}^{\ast}\succeq 0.\end{equation}
Furthermore, it follows from \eqref{dualfeasible} and \eqref{complementx2} that
\begin{equation}
\begin{bmatrix}
-{\tau}^{\ast} ~&\left(c+{g}^{\ast}\right)^{\dag}\\[3pt]
c+{g}^{\ast} ~&Q+ {\Lambda}^{\ast} \\
\end{bmatrix}
\begin{bmatrix}
1 \\[2pt]
x^{\ast} \\
\end{bmatrix}
=0,
\end{equation}
which implies $c+{g}^{\ast}=-\left(Q+ {\Lambda}^{\ast}\right)x^{\ast}.$
By this and the facts that $Q=H^{\dag} H$ and $c=-H^{\dag} r=-H^{\dag} Hx^{\ast}- H^{\dag}v,$ we have
\begin{equation}\label{Hvgx}
H^{\dag}v=g^{\ast}+{\Lambda}^{\ast} x^{\ast}.\end{equation}
Moreover, by the complementarity condition \eqref{mucomplementarity}, we obtain $\mu^{\ast}_{i,j}=0$ for $j\notin\{s_i,t_i\}$ and $\mu^{\ast}_{i,j}\geq0$ for $j\in\{s_i,t_i\}$.
This, together with the definition of $g_i$ (cf. \eqref{gi}), shows $${g}_i^{\ast}= \sum_{j=0}^{M-1} \frac{{\mu}^{\ast}_{i,j}}{2} a_j= \frac{{\mu}^{\ast}_{i,s_i}}{2} a_{s_i}+\frac{{\mu}^{\ast}_{i,t_i}}{2} a_{t_i},~i=1,2,\ldots,n.$$
Substituting the above into \eqref{Hvgx}, we immediately obtain
\begin{equation}\label{Hvi}[H^{\dag}v]_i=\lambda^\ast_i x_i^{\ast}+ \frac{\mu^{\ast}_{i,s_i}}{2} a_{s_i}+ \frac{\mu^{\ast}_{i,t_i}}{2} a_{t_i},~i=1,2,\ldots,n.\end{equation}
Let $\bar{\lambda}=\lambda^{\ast}$, $\bar{\mu}_{i,s_i}=\mu^{\ast}_{i,s_i},$ and  $\bar{\mu}_{i,t_i}=\mu^{\ast}_{i,t_i}$ for $i=1,\ldots,n$.
Hence, \eqref{Qlambda} and \eqref{Hvi} become \eqref{Qlambdabar} and \eqref{Hv}, respectively. The proof for the necessity of the condition is completed.


Next, we shall prove that the condition is sufficient too.}} Under the condition, we construct a
feasible solution $\left({\lambda}^{\ast},\,{\tau}^{\ast},\,{\mu}^{\ast}\right)$ for problem \eqref{dualproblem} as follows: \begin{equation}\label{lambda}\lambda_i^* = \bar{\lambda}_i,~i=1,2,\ldots,n\end{equation} and
\begin{equation}\label{mu}\mu^\ast_{i,j}=\left\{\begin{array}{@{}ll}
\bar{\mu}_{i,j},~~\textrm{if}~j \in \{s_i,t_i\}\\[3pt]
0,~~~~~\textrm{if} ~j \notin \{s_i,t_i\}
\end{array}\right.,~i=1,2,\ldots,n,~j=0,1,\ldots,M-1,\end{equation}
where $\bar{\lambda}_i,~\bar{\mu}_{i,s_i},$ and $\bar{\mu}_{i,t_i}$ satisfy the conditions stipulated in \eqref{existence} and \eqref{Hv}. Set
$${g}_i^{\ast}= \sum_{j=0}^{M-1} \frac{{\mu}^{\ast}_{i,j}}{2} a_j,~i=1,2,\ldots,n$$ and denote ${\Lambda}^{\ast}=\text{Diag}({\lambda}^{\ast}).$
By \eqref{Hv}, \eqref{lambda}, and \eqref{mu}, we obtain
$H^{\dag} v= {\Lambda}^* x^{\ast} + {g}^{\ast},$ which, together with the definitions of $Q$ and $c,$ further implies
$$c+{g}^*=-H^{\dag} Hx^{\ast}- H^{\dag}v+ {g}^{\ast}=-H^{\dag} Hx^{\ast}-{\Lambda}^* x^{\ast}=-\left(Q+{\Lambda}^{\ast}\right)x^{\ast}.$$
Furthermore, set
\begin{equation}\label{sigma}
{\tau}^{\ast}=\left(x^{\ast}\right)^{\dag} \left(Q+{\Lambda}^{\ast}\right)x^{\ast}+ 2\textrm{Re}\left(\left(c+{g}^{\ast}\right)^{\dag} x^{\ast} \right)
=-\left(x^{\ast}\right)^{\dag} \left(Q+{\Lambda}^{\ast}\right)x^{\ast}.
\end{equation}
Clearly, under the condition $Q+ {\Lambda}^*\succeq 0$, we have $$\begin{bmatrix}
-{\tau}^{\ast} ~&\left(c+{g}^{\ast}\right)^{\dag}\\[3pt]
c+{g}^{\ast} ~&Q+ {\Lambda}^* \\
\end{bmatrix}
=
\begin{bmatrix}
\left(x^{\ast}\right)^{\dag}\\[3pt]
-I_n\\
\end{bmatrix}
\left(Q+ {\Lambda}^{\ast}\right)
\begin{bmatrix}
x^{\ast} &-I_n \\
\end{bmatrix}
\succeq 0.$$ Therefore, the above constructed $\left({\lambda}^{\ast},\,{\tau}^{\ast},\,{\mu}^{\ast}\right)$ is
feasible to problem \eqref{dualproblem}.

Then, we show that the pairs $\left(x^{\ast},\,X^{\ast}\right)$  and $\left({\lambda}^{\ast},\,{\tau}^{\ast},\,{\mu}^{\ast}\right)$ jointly satisfy the complementarity conditions \eqref{complement1} and \eqref{complement2}.
Note that ${\mu}^{\ast}_{i,j}=0$ if $j \notin \{s_i,\,t_i\}$ (cf. \eqref{mu})
and $\cos\left(\frac{\pi}{M}\right)-\textrm{Re}(a^{\dag}_j x_i^{\ast})=0$ if $j \in \{s_i,\,t_i\}$~(from \eqref{atheta} and \eqref{ti}). Consequently, \eqref{complement1} is true. By simple calculation, we have
\begin{equation}\label{complementx}
\begin{bmatrix}
-{\tau}^{\ast} ~&\left(c+{g}^{\ast}\right)^{\dag}\\[3pt]
c+{g}^{\ast} ~&Q+ {\Lambda}^{\ast} \\
\end{bmatrix}
\bullet
\begin{bmatrix}
1 ~& \left(x^{\ast}\right)^{\dag}\\[2pt]
x^{\ast} ~& X^{\ast} \\
\end{bmatrix}=\left(x^{\ast}\right)^{\dag} \left(Q+ {\Lambda}^*\right)x^{\ast}+ 2\textrm{Re}\left((c+{g}^{\ast})^{\dag}x^{\ast}\right)-{\tau}^{\ast}=0.
\end{equation}
Therefore, both of the complementarity conditions \eqref{complement1} and \eqref{complement2} hold.
From the above analysis, we can conclude that $\left(x^{\ast},\,X^{\ast}\right)$ is an optimal solution of (CSDP2), completing
the proof that the condition is sufficient.

{{Finally, if $Q+{\bar{\Lambda}}\succ 0,$ then it follows from \cite[Theorem 10]{uniqueness} that $\left(x^{\ast},\,X^{\ast}\right)$ is the unique optimal solution of (CSDP2).}} \end{proof}

{{
Note that if $\bar{\lambda}_i >-\lambda_{\min}$ for all $i=1,2,\ldots,n$, then $Q+\textrm{Diag}(\bar{\lambda})\succ 0.$
Hence, it follows from Theorem \ref{thmgeneral} that we have the following corollary.

\begin{corollary}\label{corollary>}
Suppose that $M\geq 3.$ Let $s_i$ and $t_i$ be defined as in Theorem \ref{thmgeneral}.
If there exist \begin{equation}\bar{\lambda}_i >-\lambda_{\min},~\bar{\mu}_{i,s_i}\geq 0,~\text{and}~\bar{\mu}_{i,t_i}\geq 0,~i=1,2,\ldots,n\end{equation}
such that $H$ and $v$ in \eqref{rHv} satisfy \eqref{Hv}, then {{$\left(x^{\ast},\,X^{\ast}\right)$}} is the unique solution to (CSDP2).
\end{corollary}

As a direct consequence of Theorem \ref{thmgeneral} and Corollary 4.3,}} we have the next result. 
\begin{theorem}\label{thmM3exact}
Suppose that $M\geq 3.$ If the inputs $H$ and $v$ in \eqref{rHv} satisfy \eqref{condition}, then both (CSDP2) and (ERSDP) are tight for (P).
\end{theorem}
\begin{proof} Consider the set $$\mathcal{S}_i=\left\{ \lambda_i x_i^{\ast}+ \frac{\mu_{i,s_i}}{2} a_{s_i} + \frac{\mu_{i,t_i}}{2} a_{t_i} \in \mathbb{C}\mid\lambda_i\geq -\lambda_{\min},\,\mu_{i,s_i}\geq 0,\,\mu_{i,t_i}\geq 0 \right\}.$$
By \eqref{atheta} and \eqref{xisi}, we have $a_{s_i}=x_i^{\ast} e^{-\pi \textbf{i}/M}$ and $a_{t_i}=x_i^{\ast} e^{\pi \textbf{i}/M}$.
Hence, $\mathcal{S}_i$ can be rewritten as $\mathcal{S}_i = x_i^{\ast} \hat{\mathcal{S}}_i,$ where
$$\hat{\mathcal{S}}_i=\left\{\lambda_i + \frac{\mu_{i,s_i}}{2} e^{-\pi \textbf{i}/M} + \frac{\mu_{i,t_i}}{2} e^{\pi \textbf{i}/M} \in \mathbb{C} \mid\lambda_i\geq -\lambda_{\min},\,\mu_{i,s_i}\geq 0,\,\mu_{i,t_i}\geq 0 \right\}.$$ Note that $\hat{\mathcal{S}}_i$ is a polyhedron in the complex domain with the extreme point being $-\lambda_{\min}$ and two extreme directions being $e^{-\pi \textbf{i}/M}$ and $e^{\pi \textbf{i}/M},$ that is, a polyhedron in the $2$-dimensional real domain with the extreme point being $(-\lambda_{\min},\,0)$ and two extreme directions being $\left(\cos(\pi/M),\,-\sin(\pi/M)\right)$ and $\left(\cos(\pi/M),\,\sin(\pi/M)\right).$
One can easily verify that: 1) the origin lies in the set $\mathcal{S}_i,$ and 2) the smallest distance from the origin to the boundary of the set $\mathcal{S}_i$ is $\lambda_{\min}\sin\left(\frac{\pi}{M}\right).$ Therefore, if $H$ and $v$ in \eqref{rHv} satisfy
the condition in \eqref{condition}, then $\left[H^{\dagger}v\right]_i$ lies in the interior of $\mathcal{S}_i$ for all $i=1,2,\ldots,n,$ which further implies that
the conditions in \eqref{Hv} are satisfied. Invoking Theorem \ref{thmgeneral} and Corollary \ref{corollary>}, we conclude that (CSDP2) is tight for (P) if \eqref{condition} is satisfied. Since (ERSDP) is even tighter than (CSDP2) (cf. Theorem \ref{thmtighter}), it follows
that (ERSDP) is also tight for (P) if \eqref{condition} is satisfied. The proof is completed.
\end{proof}


Two remarks on Theorem \ref{thmM3exact} are in order.

First, Theorem \ref{thmM3exact} shows that both (CSDP2) and (ERSDP) are tight for the MIMO detection problem (P) under condition \eqref{condition}
and thus answers the open question posed by So in \cite{So2010}. Moreover, our result in Theorem \ref{thmM3exact} is for arbitrary  $M\geq 3,$ in contrast to all previous results in \cite{Jalden2006,Jalden,Kisialiou2005,Kisialiou,So2010}, which focus on the case $M=2.$

{
{
Second, Theorem 4.4 states that the true vector of transmitted signals $x^{\ast}$ can be recovered by solving appropriate SDPs if condition \eqref{condition} is satisfied.
It is worth mentioning a recent related result in \cite[Theorem 1]{Liu2017} for the general $M$-PSK modulation case,
which shows that $x^{\ast}$ can be recovered by using the generalized power method (GPM)
if the following conditions are satisfied:
\begin{equation}\label{scGPM}
\left\|\frac{2\alpha_k}{m}H^{\dag}v\right\|_{\infty}<\frac{\sin(\pi/M)}{2},~\left\|I_n-\frac{2\alpha_k}{m}H^{\dag}H\right\|_{\textrm{op}}<\frac{1}{4},
\end{equation}
where $\left\{\alpha_k\right\}$ are the step sizes used in the GPM and $\|\cdot\|_{\textrm{op}}$ is the matrix operator norm. One can show that the conditions in \eqref{scGPM}
imply 
$$\lambda_{\min}(H^{\dag}H)\sin\left(\frac{\pi}{M}\right)>\frac{3}{2}\|H^{\dag}v\|_{\infty},$$ which is more restrictive than condition \eqref{condition}. Another difference between conditions \eqref{scGPM} in \cite{Liu2017} and our condition \eqref{condition} is that conditions in \eqref{scGPM} depend on both the algorithm parameters (i.e., $\left\{\alpha_k\right\}$) and the problem parameters (i.e., $H,~v,$ and $M$) while our condition \eqref{condition} only depends on the problem parameters (and is independent of the algorithm parameters).
}}

{
In practice, $H$ and $v$ generally follow the complex Gaussian distribution.
The following theorem states how likely condition \eqref{condition} will be satisfied under the assumption that the real and imaginary parts of
the entries of $H$ and $v$ are i.i.d. real Gaussian variables. The proof of the theorem is relegated to Appendix \ref{thmprob}.
\begin{theorem}\label{thmprob}
Consider the MIMO channel model (1.1), where the entries of $\textrm{Re}(H)$ and $\textrm{Im}(H)$ are i.i.d. standard real Gaussian random variables
and the entries of $\textrm{Re}(v)$ and $\textrm{Im}(v)$ are i.i.d. real Gaussian variables with zero mean and variance $\sigma^2$.
Suppose that $m>n$ and \begin{equation}\label{sigmabound}\sigma \leq \frac{(1-\rho)^2\sin(\pi/M)}{4\sqrt{2}},\end{equation} where $\rho=\sqrt{n/m}$. Then the probability that condition \eqref{condition} is satisfied is at least$$ 1-e^{- \frac{m(1-\rho)^2}{4}}- 2\sqrt{2/\pi} \cdot n\cdot e^{-m/2}-8e^{-m/8}.$$
\end{theorem}
}
\section{Numerical experiments}

In this section, we present some {{numerical results}} to illustrate the tightness of our proposed enhanced (ERSDP) and the
conventional (CSDP) for problem (P). We generate the instances of problem (P) as follows: we generate each entry of the channel matrix
$H\in \mathbb{C}^{m\times n}$ according to the complex standard Gaussian distribution
(with zero mean and unit variance); for each $i=1,2,\ldots,n,$ we uniformly choose $s_i$ from $\{0,1,\ldots.,M-1\}$ and
set each entry of the vector $x^{\ast}\in \mathbb{C}^{n}$ to be $x_i^{\ast}=e^{2s_{i}\pi \textbf{i}/M};$
we generate each entry of the noise vector $v\in \mathbb{C}^{m}$ according to the complex Gaussian distribution
with zero mean and with variance $\sigma^2;$ and finally we compute $r$ as in \eqref{rHv}. In our numerical experiments, we set $m=15$ and $n=10.$
In practical digital communications, $M$ generally is $2,\,4,\,8,\ldots$ In our experiments, to study the tightness of (ERSDP)
on various different settings, we consider all cases where $M\in \{3,4,6,8\}.$ 

To compare the numerical performance of (ERSDP) and (CSDP), we generate a total of $40$ instances with $M=3$ and $\sigma^2\in \{0.01,\,0.1,\,1.0,\,10\}.$
Numerical results are summarized in Table 1, where ``LBC'' and ``LBE'' denote the lower bounds (i.e., the optimal objective values) returned by two SDP relaxations (CSDP) and (ERSDP), respectively; ``UB'' denotes the upper bound (i.e., the objective value at the feasible point) returned by the approximation algorithms in \cite{So2008} and \cite{Zhang}; ``GapC'' (= UB$-$LBC) and ``GapE'' (= UB$-$LBE) denote the gaps
between the upper bound and the corresponding lower bounds; ``ClosedGap'' denotes $(\textrm{LBE}-\textrm{LBC})/(\textrm{UB}-\textrm{LBC})$, which measures how much of the gap in (CSDP) is closed by (ERSDP); {{and finally the ``Y/N'' pair denotes whether (ERSDP) is tight or not and whether condition (1.5) is satisfied or not and
``Y'' and ``N'' denote ``Yes'' and ``No'', respectively.}} Obviously, the larger the value of ClosedGap is, the larger portion of the gap in (CSDP) is closed by (ERSDP) and the better the performance of (ERSDP) (compared to (CSDP)).



%

\begin{table}[!htbp]
\caption{Numerical results of (ERSDP) and (CSDP) on $40$ randomly generated instances of problem (P) with different levels of noise and $M=3.$}
\centering
\begin{tabular}{|c|c|c|c|c|c|c|c|c|}
\hline
ID &$\sigma^2$ &LBC &LBE &UB &GapC & GapE & ClosedGap &Y/N \\
\hline
1 &0.010 &-506.995 &-506.883 &-506.883 &0.111 &0.000 &100.0\%  &(Y,\,N)  \\
2 &0.010 &-332.651 &-332.539 &-332.539 &0.112 &0.000 &100.0\%  &(Y,\,Y)  \\
3 &0.010 &-253.831 &-253.568 &-253.568 &0.263 &0.000 &100.0\%  &(Y,\,Y)  \\
4 &0.010 &-194.345 &-194.280 &-194.280 &0.065 &0.000 &100.0\%  &(Y,\,Y)  \\
5 &0.010 &-217.328 &-217.190 &-217.190 &0.138 &0.000 &100.0\%  &(Y,\,Y)  \\
6 &0.010 &-344.753 &-344.581 &-344.581 &0.173 &0.000 &100.0\%  &(Y,\,Y)  \\
7 &0.010 &-212.575 &-212.430 &-212.430 &0.145 &0.000 &100.0\%  &(Y,\,Y)  \\
8 &0.010 &-159.017 &-158.900 &-158.900 &0.116 &0.000 &100.0\%  &(Y,\,Y)  \\
9 &0.010 &-206.139 &-206.054 &-206.054 &0.084 &0.000 &100.0\%  &(Y,\,Y)  \\
10 &0.010 &-228.058 &-227.955 &-227.955 &0.103 &0.000 &100.0\%  &(Y,\,Y)  \\
11 &0.100 &-133.312 &-132.217 &-132.217 &1.095 &0.000 &100.0\%  &(Y,\,N)  \\
12 &0.100 &-398.492 &-397.605 &-397.605 &0.887 &0.000 &100.0\%  &(Y,\,N)  \\
13 &0.100 &-260.164 &-259.290 &-259.290 &0.875 &0.000 &100.0\%  &(Y,\,N)  \\
14 &0.100 &-296.578 &-295.055 &-295.055 &1.523 &0.000 &100.0\%  &(Y,\,N)  \\
15 &0.100 &-202.681 &-201.966 &-201.966 &0.715 &0.000 &100.0\%  &(Y,\,N)  \\
16 &0.100 &-355.024 &-354.248 &-354.248 &0.775 &0.000 &100.0\%  &(Y,\,N)  \\
17 &0.100 &-454.615 &-453.808 &-453.808 &0.807 &0.000 &100.0\%  &(Y,\,N)  \\
18 &0.100 &-302.445 &-301.429 &-301.429 &1.016 &0.000 &100.0\%  &(Y,\,N)  \\
19 &0.100 &-392.029 &-390.804 &-390.804 &1.225 &0.000 &100.0\%  &(Y,\,N)  \\
20 &0.100 &-248.856 &-247.446 &-247.446 &1.411 &0.000 &100.0\%  &(Y,\,N)  \\
21 &1.000 &-183.786 &-171.408 &-170.839 &12.947 &0.569 &95.6\%  &(N,\,N)  \\
22 &1.000 &-388.290 &-374.763 &-370.751 &17.539 &4.012 &77.1\%  &(N,\,N)  \\
23 &1.000 &-309.820 &-296.907 &-296.907 &12.913 &0.000 &100.0\%  &(Y,\,N)  \\
24 &1.000 &-181.890 &-157.968 &-151.285 &30.605 &6.683 &78.2\%  &(N,\,N)  \\
25 &1.000 &-235.490 &-219.367 &-217.811 &17.679 &1.557 &91.2\%  &(N,\,N)  \\
26 &1.000 &-263.610 &-249.803 &-249.803 &13.807 &0.000 &100.0\%  &(Y,\,N)  \\
27 &1.000 &-373.140 &-361.259 &-361.259 &11.882 &0.000 &100.0\%  &(Y,\,N)  \\
28 &1.000 &-214.068 &-201.188 &-201.188 &12.879 &0.000 &100.0\%  &(Y,\,N)  \\
29 &1.000 &-404.825 &-398.641 &-398.564 &6.260 &0.077 &98.8\%  &(N,\,N)  \\
30 &1.000 &-343.993 &-339.752 &-339.752 &4.241 &0.000 &100.0\%  &(Y,\,N)  \\
31 &10.000 &-390.985 &-300.162 &-282.554 &108.431 &17.608 &83.8\%  &(N,\,N)  \\
32 &10.000 &-298.119 &-221.992 &-205.475 &92.644 &16.517 &82.2\%  &(N,\,N)  \\
33 &10.000 &-386.464 &-329.595 &-328.439 &58.025 &1.156 &98.0\%  &(N,\,N)  \\
34 &10.000 &-356.127 &-314.399 &-304.091 &52.036 &10.308 &80.2\%  &(N,\,N)  \\
35 &10.000 &-329.525 &-243.219 &-212.705 &116.820 &30.514 &73.9\%  &(N,\,N)  \\
36 &10.000 &-324.014 &-284.622 &-269.178 &54.836 &15.444 &71.8\%  &(N,\,N)  \\
37 &10.000 &-548.569 &-477.101 &-468.082 &80.487 &9.019 &88.8\%  &(N,\,N)  \\
38 &10.000 &-787.389 &-737.455 &-737.359 &50.029 &0.095 &99.8\%  &(N,\,N)  \\
39 &10.000 &-366.880 &-307.478 &-305.544 &61.336 &1.935 &96.8\%  &(N,\,N)  \\
40 &10.000 &-377.069 &-304.619 &-280.803 &96.266 &23.816 &75.3\%  &(N,\,N)  \\
\hline
\end{tabular}
\end{table}

We can observe from Table 1 that (ERSDP) is generally tight for problem (P) when the noise level $\sigma^2$ is low, i.e., when the SNR is high. More specifically, (ERSDP) is tight for {all 20 instances with $\sigma^2\leq 0.1$ and for $5$ instances with $\sigma^2=1$}. In sharp contrast, (CSDP) is not tight for all
instances. Furthermore, we can see from Table 1 that LBE is much larger than LBC and GapC is much larger than GapE for instances with high levels of noise,
which show that (ERSDP) is much tighter than (CSDP) for instances with high levels of noise. For instance, when $\sigma^2 =10$, GapE is generally
much smaller than GapC and about {$70\%$ to $99.8\%$} of the gap in (CSDP) is closed by (ERSDP).

{{
To study the  empirical probability of (ERSDP) being tight (as a function of $\sigma^{2}$),
we generate more problem instances where $M=3$ and $\sigma^2\in [0.001,10].$ For each setting, we randomly
generate $1000$ instances. Figure 2 illustrates the empirical probability of (ERSDP) being
tight and the empirical probability of condition (1.5) being satisfied. We can conclude from Fig. 2 that: 1) for the case where $m=15,~n=10,$ and $M=3$, the probability of (ERSDP)
being tight is very close to one when $\sigma^2\leq 0.1$; 2) since condition (1.5) is a sufficient condition for (ERSDP) to
be tight, the probability of condition (1.5) being satisfied is lower than the probability of (ERSDP) being
tight.
}}

\begin{figure}[!t]
\centering
\includegraphics[width=9.0cm]{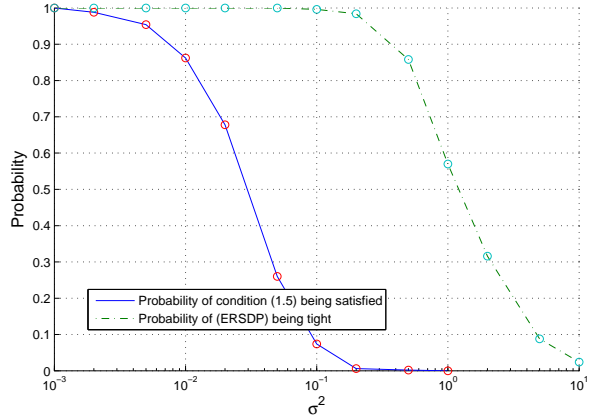}
\caption{ {{Empirical probability of (ERSDP) being
tight and condition (1.5) being satisfied.}} }
\end{figure}


To gain more insight towards the numerical performance of (ERSDP), we carry out numerical experiments on more problem instances with different $M\in\{4,\,6,\,8\}$ and different $\sigma^2\in\{0.01,\,0.1,\,1.0,\,10\}$.
More specifically, we test the performance of (ERSDP) on $12$ different settings and for each setting
we randomly generate $100$ instances. 
Table 2 summarizes the average results (over the $100$ problem instances) and the statistical results (computed based on the $100$ problem instances), where ``TimeC'' and ``TimeE''
denote the average CPU time for solving (CSDP) and (ERSDP), respectively; ``ProbC'' and ``ProbE'' denote the probability that (CSDP) and (ERSDP) are tight, respectively.

\begin{table}[t!]
\caption{Numerical results of (ERSDP) and (CSDP) on randomly generated instances of problem (P) with different $M$ and different levels of noise.}
\centering
\begin{tabular}{|c|c|c|c|c|c|c|c|c|}
\hline
$M$ &$\sigma^2$ &GapC & GapE &ClosedGap &TimeC &TimeE &ProbC &ProbE \\
\hline
4 &$0.01$ &0.096 &0.000  &100.0\% &0.05 &0.08  &0\%  &100\%  \\
6 &$0.01$ &0.106 &0.000  &100.0\% &0.05 &0.08  &0\%  &100\%  \\
8 &$0.01$ &0.100 &0.000  &99.9\% &0.05 &0.09  &0\%  &99\%    \\
4 &$0.1$ &0.971 &0.000  &100.0\% &0.05 &0.08  &0\%  &100\%   \\
6 &$0.1$ &0.871 &0.015  &98.9\% &0.05 &0.08  &0\%  &80\%     \\
8 &$0.1$ &0.979 &0.056  &96.2\% &0.05 &0.09  &0\%  &59\%     \\
4 &$1.0$ &9.399 &0.934  &91.0\% &0.05 &0.07  &0\%  &32\%     \\
6 &$1.0$ &9.423 &2.646  &75.7\% &0.05 &0.07  &0\%  &7\%     \\
8 &$1.0$ &9.162 &3.739  &60.8\% &0.05 &0.07  &0\%  &1\%     \\
4 &$10$ &54.251 &11.328  &78.6\% &0.05 &0.07  &0\%  &0\%    \\
6 &$10$ &26.827 &8.085  &68.2\% &0.05 &0.07  &0\%  &0\%     \\
8 &$10$ &14.803 &4.617  &68.0\% &0.05 &0.07  &0\%  &0\%     \\
\hline
\end{tabular}
\end{table}


We can conclude from Table 2 that: 1) the probability that (ERSDP) is tight for (P) is very high if the
noise level is low; 2) if the noise level is high, (ERSDP) is much tighter than (CSDP) and (ERSDP) narrows down over $60\%$ of the gap due to (CSDP); and 3) the CPU time of solving (ERSDP) is not much longer
than the one of solving (CSDP). 
Finally, (ERSDP) performs better in terms of the tightness on instances with smaller $M$ (if the noise level is fixed).
This is consistent with the sufficient condition in \eqref{condition}, which shows that a lower level of the noise is required to guarantee the
tightness of (ERSDP) for problem (P) with a larger $M.$ 


%
%

\section{Conclusion}
In this paper, we considered the MIMO detection problem (P), an important class of quadratic programming problems with
unit-modulus and discrete argument constraints. We showed that the conventional (CSDP) is generally
not tight for (P). Moreover, we proposed a new enhanced SDR called (ERSDP) and showed that (ERSDP)
is tight for problem (P) under the condition in \eqref{condition}. 
~To the best of our knowledge, our proposed (ERSDP) is the first SDR that is theoretically guaranteed to be tight for general cases of problem (P). 
Our above results answered an open question posed by So in \cite{So2010}.
In addition to enjoying strong theoretical guarantees, our proposed (ERSDP) also performs very well numerically.
Our {numerical} results show that (ERSDP) can return the true vector of transmitted signals with high probability if the level of the noise is low and 
it can narrow down more than $60\%$ of the gap due to (CSDP) on average if the level of the noise is high. {{It is interesting to compare the tightness of some existing SDRs (e.g., in \cite{Fan,Mobasher}) with our proposed (ERSDP) and (CSDP2).
If we could show that some existing SDRs are stronger/tighter than (CSDP2), then it follows from this and Theorems \ref{thmgeneral} and \ref{thmM3exact} that condition \eqref{condition} is also a sufficient condition for them to be tight.}}
We believe that our proposed (ERSDP) and related techniques will find further applications such as in the development of quantized precoding for massive multi-user MIMO communications \cite{jacobsson2017quantized,Sohrabi2018}.

{{\section*{Acknowledgement}
The authors would like to thank Professor Florian Jarre for many useful discussions on the paper.
We also thank the associate editor,
Professor Anthony So, and two anonymous referees for their detailed and valuable
comments and suggestions, which significantly improved
the quality of the paper.}}

%
%
%

\section*{Appendix A: Equivalence between (CSDP) and (RSDP)}

We first show that, for any feasible point $(x,\,X)$ of (CSDP), we can construct a feasible point
$(y,\,Y)$ of (RSDP) that satisfies $\hat{Q}\bullet Y+2 \hat{c}^{\T}y=Q\bullet X+ 2\textrm{Re} (c^{\dag} x).$
Since $$\begin{bmatrix}
1 ~&x^{\dag}\\[2pt]
x ~&X
\end{bmatrix}\succeq 0,$$ we assume, without loss of generality, that it has the following decomposition:
\begin{equation}\label{Xdecom}
\begin{bmatrix}
1 ~&x^{\dag}\\[2pt]
x ~&X
\end{bmatrix}=\sum_{j=1}^{n+1}
\begin{bmatrix}
t_j\\[2pt]
v_j
\end{bmatrix}
\begin{bmatrix}
t_j\\[2pt]
v_j
\end{bmatrix}^{\dag},
\end{equation}
where $t_j\in \mathbb{R}_{+}$ and $v_j\in \mathbb{C}^n$ for all $j=1,2,\ldots,n+1$.
Let
\begin{equation}\label{yY}y= \sum_{j=1}^{n+1} t_j \begin{bmatrix}
\textrm{Re}(v_j)\\[2pt]
\textrm{Im}(v_j)
\end{bmatrix}~\text{and}~Y=\sum_{j=1}^{n+1}
\begin{bmatrix}
\textrm{Re}(v_j)\\[2pt]
\textrm{Im}(v_j)
\end{bmatrix}
\begin{bmatrix}
\textrm{Re}(v_j)\\[2pt]
\textrm{Im}(v_j)
\end{bmatrix}^{\T}.\end{equation} It is simple to see that
$$
\begin{bmatrix}
1 ~&y^{\T}\\[2pt]
y ~&Y
\end{bmatrix}=\sum_{j=1}^{n+1}
\begin{bmatrix}
t_j\\[2pt]
\textrm{Re}(v_j)\\[2pt]
\textrm{Im}(v_j)
\end{bmatrix}
\begin{bmatrix}
t_j\\[2pt]
\textrm{Re}(v_j)\\[2pt]
\textrm{Im}(v_j)
\end{bmatrix}^{\T} \succeq 0.
$$ It follows from \eqref{Xdecom} and \eqref{yY} that, for any $i=1,2,\ldots,n,$ we have
$$Y_{i,i}+Y_{n+i,n+i}=\sum_{j=1}^{n+1}\left(\left[\textrm{Re}(v_j)\right]_{i}^2+\left[\textrm{Im}(v_j)\right]_{i}^2 \right)=\sum_{j=1}^{n+1}\left|\left[v_j\right]_i\right|^2=X_{i,i}=1.$$
Hence, $(y,\,Y)$ in \eqref{yY} is feasible to (RSDP). Moreover, we have, from \eqref{Xdecom} and \eqref{yY}, that
\begin{align*}
\hat{Q}\bullet Y+2 \hat{c}^{\T}y
&=\hat{Q}\bullet \left(\sum_{j=1}^{n+1}\begin{bmatrix}
\textrm{Re}(v_j)\\[2pt]
\textrm{Im}(v_j)
\end{bmatrix}
\begin{bmatrix}
\textrm{Re}(v_j)\\[2pt]
\textrm{Im}(v_j)
\end{bmatrix}^{\T} \right)+ 2\hat{c}^{\T} \left(\sum_{j=1}^{n+1} t_j\begin{bmatrix}
\textrm{Re}(v_j)\\[2pt]
\textrm{Im}(v_j)
\end{bmatrix}\right)\\
&=Q\bullet\left(\sum_{j=1}^{n+1} v_j v_j^{\dag}\right)+2\textrm{Re}\left(c^{\dag} \left(\sum_{j=1}^{n+1} t_j v_j\right)\right)\\
&=Q\bullet X+ 2\textrm{Re} (c^{\dag} x).
\end{align*}
We can conclude from the above that (CSDP) is at least as tight as (RSDP).

Conversely, for any given feasible point $(y,\,Y)$ of (RSDP), we can construct a feasible point $(x,\,X)$ of (CSDP) that gives the same objective value, i.e., $Q\bullet X+ 2\textrm{Re} (c^{\dag} x)=\hat{Q}\bullet Y+2 \hat{c}^{\T}y.$
Suppose that
$$
y=\begin{bmatrix}
a\\[2pt]
b\\
\end{bmatrix}~\textrm{and}~Y=\begin{bmatrix}
A ~&B\\[2pt]
B^{\T} ~&C\\
\end{bmatrix},$$
where $a,\,b\in\mathbb{R}^n$, $A,\,B,\,C\in \mathbb{R}^{n\times n}$, $A=A^{\T}$, and $C=C^{\T}.$ Let
\begin{equation}\label{xX}
x=a+b \textbf{i}~\text{and}~X=(A+C) + (B^{\T}-B)\textbf{i}.
\end{equation} Next, we first show \begin{equation}\label{hatX}\hat X:=\begin{bmatrix}
1 ~&x^{\dag}\\[2pt]
x ~&X
\end{bmatrix}\succeq 0.\end{equation}
Define
$$M_1=\begin{bmatrix}
1 ~&a^{\T} ~&0 ~&b^{\T}\\[2pt]
a ~&A  ~&\zeta ~&B\\[2pt]
0 ~&\zeta^{\T}  ~&0 ~&\zeta^{\T}\\[2pt]
b ~&B^{\T}  ~&\zeta ~&C
\end{bmatrix}~\text{and}~M_2=\begin{bmatrix}
0 &\zeta^{\T} &0 &\zeta^{\T}\\
\zeta &C  &-b &-B^{\T}\\
0 &-b^{\T} &1 &a^{\T}\\
\zeta &-B   &a &A
\end{bmatrix},$$ where $\zeta$ is the all-zero column vector of dimension $n$.
Then, it follows from $$\begin{bmatrix}1 ~&y^{\T}\\[2pt]
y ~&Y\\
\end{bmatrix}\succeq 0,$$ that $$M_1\succeq 0~\text{and}~M_2 =
\begin{bmatrix}
 ~&I_{n+1} \\[3pt]
-I_{n+1} ~&
\end{bmatrix} M_1
\begin{bmatrix}
 ~&-I_{n+1} \\[3pt]
I_{n+1} ~&
\end{bmatrix}
\succeq 0,$$ which further implies
\begin{equation}\label{M}M:=M_1+M_2=\begin{bmatrix}
1 ~&a^{\T} ~&0 ~&b^{\T}\\[2pt]
a ~&A+C  ~&-b ~&B-B^{\T}\\[2pt]
0 ~&-b^{\T}  ~&1 ~&a^{\T}\\[2pt]
b ~&B^{\T}-B  ~&a ~&A+C
\end{bmatrix}\succeq 0.\end{equation}
Recall that $v\in \mathbb{C}^{n+1}$ is an eigenvector of $\hat X$ in \eqref{hatX} if and only if
$$
\begin{bmatrix}
\textrm{Re}(v)\\[2pt]
\textrm{Im}(v)\\
\end{bmatrix}
\textrm{ and }
\begin{bmatrix}
-\textrm{Im}(v)\\[2pt]
\textrm{Re}(v)\\
\end{bmatrix}
$$ are two eigenvectors of $M$ in \eqref{M} corresponding to the same eigenvalue \cite[Page 448]{Goemans2004}. 
Then, we have $\hat X \geq 0.$
The positive semidefiniteness of $\hat X$ further implies that all of its diagonal entries are nonnegative. This, together with \eqref{xX}, immediately shows
that $X_{i,i}=1$ for all $i=1,2,\ldots,n.$ Moreover, it is simple to check that $x$ and $X$ in \eqref{xX} satisfy
$Q\bullet X +2 \textrm{Re}(c^{\dag}x)=\hat{Q} \bullet Y +2\hat{c}^{\T}y$. This shows that (RSDP) is at least as tight as (CSDP).

From the above analysis, we conclude that (CSDP) are (RSDP) are equivalent to each other.

\section*{Appendix B: Proof of Theorem \ref{thmtighter}}

We first show that, for any
feasible point $(y,\,Y)$ of (ERSDP), we can construct a feasible point
$(x,\,X)$ of (CSDP2) that satisfies $Q\bullet X+ 2\textrm{Re} (c^{\dag} x)=\hat{Q}\bullet Y+2 \hat{c}^{\T}y.$
Similar to the proof in Appendix A, 
we can construct the same $x$ and $X$ in \eqref{xX} and show that $\hat X$ in \eqref{hatX} satisfies $\hat X\succeq 0,$
$X_{i,i}=1$ for all $i=1,2,\ldots,n,$ and $Q\bullet X +2 \textrm{Re}(c^{\dag}x)=\hat{Q} \bullet Y +2\hat{c}^{\T}y.$
For each $i=1,2,\ldots,n$, since $\mathcal{Y}(i) \in \textrm{Conv} \{P_0,P_1,\ldots,P_{M-1}\}$ in (ERSDP),
it follows that $$\left(y_i,\,y_{i+n}\right)\in \textrm{Conv}\left(\left\{(\cos\theta,\,\sin\theta)\mid\theta\in \mathcal{A}\right\}\right),$$
which is equivalent to
$$\textrm{Re}(a^{\dag}_j x_i) \leq  \cos\left(\frac{\pi}{M}\right),~j=0,1,\ldots,M-1,$$
where $x_i=y_i+y_{i+n}\textbf{i}$ and $a_j$ is given in \eqref{atheta}. Hence, $(x,\,X)$ is feasible to (CSDP2).
Based on the above analysis, we conclude that (ERSDP) is at least as tight as (CSDP2).

Next, we give an example to illustrate that (ERSDP) is indeed (strictly) tighter than (CSDP2).
Let $m=n=2$ and $M=3$ (and hence the set $\mathcal{A}$ in \eqref{setA} is $\left\{0, 2\pi/3, 4\pi/3\right\}$ in this case);
let
$$x^{\ast}=\left(
        \begin{array}{c}
          \frac{-1-\sqrt{3}\mathbf{i}}{2} \\[3pt]
          1 \\
        \end{array}
      \right),~H=\left(
                   \begin{array}{cc}
                     8-6\mathbf{i} ~& 8+6\mathbf{i} \\[3pt]
                     3+4\mathbf{i} ~& -4-3\mathbf{i} \\
                   \end{array}
                 \right),~\text{and}~v=\left(
                                        \begin{array}{c}
                                          5-6\mathbf{i} \\
                                          4-4\mathbf{i} \\
                                        \end{array}
                                      \right);
$$and finally set $r=Hx^{\ast}+v$ as in \eqref{rHv}. Then, we numerically solve (CSDP), (CSDP2), and (ERSDP), and their optimal objective values are $$-76.3176,~-45.1273,~\text{and}~-25.4763,$$ respectively. Hence, (CSDP2) is tighter (CSDP) and (ERSDP) is further tighter than (CSDP2).
%

\section*{Appendix C: Derivation of the Dual Problem of (CSDP2)}
The canonical form of (CSDP2) is
\begin{equation*}\label{canonial}\begin{array}{cl}
\displaystyle \min_{y,\,\hat X} ~&~ \hat C \bullet \hat X  \\[2pt]
\mbox{s.t.} 
            ~&~ E_{i}  \bullet \hat X =1,~i=1,2,\ldots,n+1,\\[3pt]
            ~&~\textrm{Re}(a^{\dag}_j \hat X_{i+1,1}) + y_{i,j} =\cos\left(\frac{\pi}{M}\right),~i=1,2,\ldots,n,~j=0,1,\ldots,M-1,\\[6pt]
            ~&~y \geq 0,~\hat X\succeq 0,
\end{array}\end{equation*}
where the variables $y\in\mathbb{R}^{M\times n}$ and $\hat X \in \mathbb{C}^{(n+1) \times (n+1)},$
$\hat C = \left(
        \begin{array}{cc}
          0 & ~c^\dagger \\
          c & ~Q \\
        \end{array}
      \right),$ and $E_{i}\in \mathbb{R}^{(n+1)\times(n+1)}$ is the all-zero matrix except its $i$-th diagonal entry being $1.$ 
   Let $\tau$ be the Lagrange multiplier corresponding to the constraint $E_1 \bullet \hat X =1,$ $-\lambda_i$
   be the Lagrange multiplier corresponding to the constraint $E_{i+1}  \bullet \hat X =1,$ and $-\mu_{i,j}$ be the Lagrange multiplier corresponding to the constraint $\textrm{Re}\left(a^{\dag}_j \hat X_{i+1,1}\right) + y_{i,j} =\cos\left(\frac{\pi}{M}\right).$ Notice that both the nonnegative orthant cone and the semidefinite cone are self-dual. Thus, \eqref{dualproblem} is the dual problem of (CSDP2). 

\section*{Appendix D: Proof of Theorem \ref{thmprob}}
{{
To prove Theorem \ref{thmprob}, we first state the following two lemmas.

\begin{lemma}
Under the same assumptions as in Theorem \ref{thmprob}, the smallest singular value $s_{\min}$ of the matrix $\frac{1}{\sqrt{2m}} H$ satisfies
\begin{equation}\label{probability1}\mathbb{P}\left(s_{\min}\leq 1-\sqrt{n/m}-t\right)\leq e^{-mt^2}.\end{equation}
\end{lemma}

For the proof of the above lemma, we refer the reader to \cite{Davidson} and \cite[Chapter 9]{Foucart}. 
Substituting $\rho=\sqrt{n/m}$ and $t=(1-\rho)/2$ into \eqref{probability1}, we get
$$\mathbb{P}\left(s_{\min}\leq (1-\rho)/2\right)\leq e^{- \frac{m(1-\rho)^2}{4}},$$
which, together with the fact $\lambda_{\min}(H^\dag H) = 2m s^2_{\min},$ further implies
\begin{equation}\label{Prob0}
\mathbb{P}\left(\lambda_{\min}(H^\dag H) \leq m(1-\rho)^2/2\right)\leq  e^{- \frac{m(1-\rho)^2}{4}}.
\end{equation}



The next result is an extension of \cite[Proposition 3.6]{So2010}.
\begin{lemma}
Under the same assumptions as in Theorem \ref{thmprob}, there holds
\begin{equation}\label{Hvsigmaprob}
\mathbb{P}\left(\|H^{\dag}v\|_{\infty} > 2\sqrt{2} m \sigma \right)\leq 2\sqrt{2/\pi} \cdot n\cdot e^{-m/2}+8e^{-m/8}.
\end{equation}
\end{lemma}
\begin{proof}
Let
\begin{align*}
\hat{H}=\left[\hat H_1, \hat H_2
        \right],~\hat H_1=\left[
   \begin{array}{c}
     \textrm{Re}(H) \\
     \textrm{Im}(H) \\
   \end{array}
 \right],~\hat H_2=\left[
            \begin{array}{c}
              -\textrm{Im}(H) \\
              \textrm{Re}(H) \\
            \end{array}
          \right],
\text{and}~\hat{v}=\begin{bmatrix}
\textrm{Re}(v) \\[3pt]
\textrm{Im}(v)\\
\end{bmatrix}.
\end{align*}
It follows from \cite[Proposition 3.6]{So2010} that
\begin{equation*}
\mathbb{P}\left( \|\hat{H}_1^{\T} \hat{v}\|_{\infty} \geq 2 m \sigma\right) \leq \sqrt{2/\pi} \cdot n\cdot e^{-m/2}+4e^{-m/8}
\end{equation*}
and
\begin{equation*}
\mathbb{P}\left( \|\hat{H}_2^{\T} \hat{v}\|_{\infty} \geq 2 m \sigma\right) \leq \sqrt{2/\pi} \cdot n\cdot e^{-m/2}+4e^{-m/8},
\end{equation*}
and thus
\begin{equation*}
\mathbb{P}\left( \|\hat{H}^{\T} \hat{v}\|_{\infty} \geq 2 m \sigma\right)  \leq 2 \sqrt{2/\pi} \cdot n\cdot e^{-m/2}+8e^{-m/8}.
\end{equation*}
Since $ \|\hat{H}^{\T} \hat{v}\|_{\infty}\geq  \frac{\sqrt{2}}{2}  \|H^{\dag}v\|_{\infty},$ we immediately get the desired result \eqref{Hvsigmaprob}.~ %
\end{proof}

%
Now we are ready to prove Theorem \ref{thmprob}. The proof is similar to the one in \cite[Theorem 3.3]{So2010}.
First, combining \eqref{sigmabound} and \eqref{Hvsigmaprob}, we have
\begin{equation}\label{Prob2}
\mathbb{P}\left\{ \|H^{\dag}v\|_{\infty} \geq \frac{ m(1-\rho)^2 \sin(\pi/M)}{2} \right\} \leq 2\sqrt{2/\pi} \cdot n\cdot e^{-m/2}+8e^{-m/8}.
\end{equation}
Then, it follows from \eqref{Prob0} and \eqref{Prob2} that the probability that the event $$\left\{ \lambda_{\min}(H^\dag H) \geq \frac{m(1-\rho)^2}{2}\right\} \bigcap \left\{ \|H^{\dag}v\|_{\infty} \leq \frac{ m(1-\rho)^2 \sin(\pi/M)}{2}\right\}$$ happens is at least $1-e^{- \frac{m(1-\rho)^2}{4}}- 2\sqrt{2/\pi} \cdot n\cdot e^{-m/2}-8e^{-m/8}.$
Finally, it is simple to verify that the probability that the above event happens is greater than or equal to the one that condition \eqref{condition} is satisfied. The proof of Theorem \ref{thmprob} is completed.

%

%

}}




\begin{thebibliography}{}


\bibitem{uniqueness}
{\sc F. Alizadeh, J.-P. A. Haeberly, and M. L. Overton}, {\em Complementarity
and nondegeneracy in semidefinite programming}, Math. Program., 77 (1997), pp.~111--128.


\bibitem{Bandeira2016}
{\sc A. S. Bandeira, N. Boumal, and A. Singer}, {\em Tightness of the maximum likelihood
  semidefinite relaxation for angular synchronization}, Math. Program., 163 (2017), pp.~145--167.

\bibitem{Bandeira2014}
{\sc A. S. Bandeira, Y. Khoo, and A. Singer}, {\em Open problem: Tightness of maximum likelihood semidefinite relaxations},
J. Mach. Learn. Res.: Workshop and Conference Proceedings, 35 (2014), pp.~1265--1267.


\bibitem{Beck}
{\sc A. Beck and M. Teboulle}, {\em Global optimality conditions for quadratic optimization problems
with binary constraints}, SIAM J. Optim., 11 (2000), pp.~179--188.

\bibitem{Boumal}
{\sc N. Boumal}, {\em Nonconvex phase synchronization}, SIAM J. Optim., 26 (2016), pp.~2355--2377.

{{
\bibitem{Damen}
{\sc O. Damen, A. Chkeif, and J.-C. Belfiore}, {\em Lattice code decoder for space-time codes}, IEEE Commun. Lett., 4 (2000), pp.~161--163.

\bibitem{Davidson}
{\sc K. R. Davidson and S. J. Szarek}, {\em Local operator theory, random matrices
and Banach spaces}, in Handbook of the geometry of Banach spaces, Vol. I,  2001, Amsterdam: North-Holland, pp.~317--366.
}}




{{
\bibitem{Fan}
{\sc X. Fan, J. Song, D. P. Palomar, and O. C. Au}, {\em Universal binary semidefinite relaxation for ML signal
detection}, IEEE Trans. Commun., 61 (2013), pp.~4565--4576.


\bibitem{Foucart}
{\sc F. S. Foucart and H. Rauhut}, {\em A Mathematical Introduction to Compressive Sensing}, Birkh\"{a}user, Basel, Switzerland, 2013.
}}

\bibitem{Goemans1995}
{\sc M. Goemans and D. Williamson}, {\em Improved approximation algorithms for maximum cut and satisfiability problems using semidefinite programming},
Journal of the ACM, 42 (1995), pp.~1115--1145.

\bibitem{Goemans2004}
{\sc M. Goemans and D. Williamson}, {\em Approximation algorithms for Max-3-Cut and other
problems via complex semidefinite programming}, J. Comput. Syst. Sci., 68 (2004), pp.~442--470.

{{
\bibitem{Honig}
{\sc M. Honig, U. Madhow, and S. Verdu}, {\em Blind adaptive multiuser detection}, IEEE Trans. Inf. Theory, 41 (1995), pp.~944--960.
}}

\bibitem{jacobsson2017quantized}
{\sc S. Jacobsson, G. Durisi, M. Goldstein, and C. Studer}, {\em Quantized precoding for massive MU-MIMO}, IEEE Trans. Commun., 65 (2017), pp.~4670-4684.


\bibitem{Jalden2006}
{\sc J. Jald\'{e}n}, {\em Detection for multiple input multiple output channels}, Ph.D. Thesis, School of
Electrical Engineering, KTH, Stockholm, Sweden, 2006.

\bibitem{Jalden}
{\sc J. Jald\'{e}n, C. Martin, and B. Ottersten}, {\em Semidefinite programming for
  detection in linear systems -- Optimality conditions and space-time decoding}, in Proceedings of the IEEE International Conference on Acoustics, Speech, and Signal Processing (ICASSP'03), Hong Kong, 2003, pp.~9--12.

{{
\bibitem{Jalden2008}
{\sc J. Jald\'{e}n and B. Ottersten}, {\em The diversity order of the semidefinite relaxation detector}, IEEE Trans. Inf. Theory, 54 (2008), pp.~1406--1422.
}}


\bibitem{Kisialiou2005}
{{
{\sc M. Kisialiou and Z.-Q. Luo}, {\em Performance analysis of quasi-maximum-likelihood detector based on semi-definite programming}, in
Proceedings of the IEEE International Conference on Acoustics, Speech, and Signal Processing (ICASSP'05), Philadelphia, 2005, pp.~433--436.
}}



\bibitem{Kisialiou}
{\sc M. Kisialiou and Z.-Q. Luo}, {\em Probabilistic analysis of semidefinite relaxation for binary quadratic minimization}, SIAM J. Optim., 20 (2010), pp.~1906--1922.

{{
\bibitem{Kohno}
{\sc R. Kohno, H. Imai, and M. Hatori}, {\em Cancellation techniques of co-channel interference in asynchronous spread spectrum multiple access
systems}, Trans. Elect. Commun., 66 (1983), pp.~416--423.
}}


{{
\bibitem{Lemarechal}
{\sc C. Lemarechal and F. Oustry}, {\em SDP relaxations in
combinatorial optimization from a Lagrangian point of view}, in Advances in Convex Analysis and Global Optimization, N. Hadijsavvas
and P.M. Paradalos, eds., Kluwer, Norwell, MA, 2001, pp.~119--134.
}}

\bibitem{Liu}
{\sc H. Liu, M.-C. Yue, and A. M.-C. So}, {\em On the estimation performance and convergence rate of the generalized power method for phase synchronization}, SIAM J. Optim., 27 (2017), pp.~2426-2446.

\bibitem{Liu2017}
{\sc H. Liu, M.-C. Yue, and A. M.-C. So}, {\em A discrete first-order method for large-scale MIMO detection with provable guarantees},
in Proceedings of the 18th IEEE Workshop on Signal Processing Advances in Wireless Communications (SPAWC'17), Sapporo, 2017, pp.~669--673.


\bibitem{simo}
{\sc Y.-F. Liu, M. Hong, and Y.-H. Dai}, {\em Max-min fairness linear transceiver design problem for a multi-user SIMO interference channel is
polynomial time solvable}, IEEE Signal Process. Lett., 20 (2013), pp.~27--30.

\bibitem{luliu2017tsp}
{\sc C. Lu and Y.-F. Liu}, {\em An efficient global algorithm for single-group multicast beamforming}, IEEE Trans. Signal Process., 65 (2017), pp.~3761--3774.

\bibitem{Lu2017iccc}
{\sc C. Lu, Y.-F. Liu, and J. Zbou}, {\em An efficient global algorithm for nonconvex complex quadratic problems with applications in wireless communications},
in Proceedings of the 6th IEEE/CIC International Conference on Communications in China (ICCC'17), Qingdao, 2017, pp.~1--5.

\bibitem{Luo2010}
{\sc Z.-Q. Luo, W.-K. Ma, A. M.-C. So, Y. Ye, and S. Zhang}, {\em Semidefinite relaxation of
  quadratic optimization problems}, IEEE Signal Process. Mag., 27 (2010), pp.~20--34.


\bibitem{Ma2004}
{\sc W.-K. Ma, P.-C. Ching, and Z. Ding}, {\em Semidefinite relaxation based multiuser
detection for M-ary PSK multiuser systems}, IEEE Trans. Signal Process., 52 (2004), pp.~2862--2872.

\bibitem{Maio2009}
{\sc A. D. Maio, S. D. Nicola, Y. Huang, Z.-Q. Luo, and S. Zhang}, {\em Design of phase codes for radar
performance optimization with a similarity constraint}, IEEE Trans. Signal Process., 57 (2009), pp.~610--621.

{{
\bibitem{Mobasher}
{\sc A. Mobasher, M. Taherzadeh, R. Sotirov, and A. K. Khandani}, {\em A near-maximum-likelihood decoding algorithm
for MIMO systems based on semi-definite programming}, IEEE Trans. Inf. Theory, 53 (2007), pp.~3869--3886.

\bibitem{Murugan}
{\sc A. D. Murugan, H. E. Gamal, M. O. Damen, and G. Caire}, {\em A unified framework for tree search decoding: Rediscovering the sequential
decoder}, IEEE Trans. Inf. Theory, 52 (2006), pp.~933--953.
}}


\bibitem{Wenqiang17}
{\sc W. Pu, Y.-F. Liu, J. Yan, S. Zhou, H. Liu, and Z.-Q. Luo}, {\em Optimal estimation of sensor biases for asynchronous multi-sensor data fusion}, Math. Program.
170 (2018), pp.~357--386.

%


{{
\bibitem{Schneider}
{\sc K. S. Schneider}, {\em Optimum detection of code division multiplexed
signals}, IEEE Trans. Aerosp. Elect. Syst., AES-15 (1979), pp.~181--185.
}}

{{
\bibitem{Shor}
{\sc N. Z. Shor and A. S. Davydov}, {\em Method of obtaining estimates in quadratic extremal problems with Boolean variables}, Cybern. Syst. Anal., 21 (1985), pp.~207--210.
}}

\bibitem{Singer}
{\sc A. Singer}, {\em Angular synchronization by eigenvectors and semidefinite programming}, Appl. Comput. Harmon. Anal., 30 (2011), pp.~20--36.

\bibitem{So2010}
{\sc A. M.-C. So}, {\em Probabilistic analysis of the semidefinite relaxation detector
  in digital communications}, in Proceedings of the Twenty-First Annual ACM-SIAM Symposium on Discrete Algorithms (SODA'10), Austin, 2010, pp.~698--711.

\bibitem{So2008}
{\sc A. M.-C. So, J. Zhang, and Y. Ye }, {\em On approximating complex quadratic optimization
problems via semidefinite programs}, Math. Program., 110 (2007), pp.~93--110.



\bibitem{Sohrabi2018}
{\sc F. Sohrabi, Y.-F. Liu, and W. Yu}, {\em One-bit precoding and constellation range design for massive {MIMO} with {QAM} signaling}, IEEE J. Sel. Topics Signal Process., 12 (2018), pp.~557--570.

\bibitem{Soltanalian}
{\sc M. Soltanalian and P. Stoica}, {\em Designing unimodular codes via quadratic optimization}, IEEE Trans. Signal Process., 62 (2014), pp.~1221--1234.

\bibitem{Sun}
{\sc J. Sun, Q. Qu, and J. Wright}, {\em When are nonconvex problems not scary}, preprint, arXiv:1510.06096, 2016.

\bibitem{Tan}
{\sc P. H. Tan and L. K. Rasmussen}, {\em The application of semidefinite programming for
  detection in CDMA}, IEEE J. Sel. Areas Commun., 19 (2001), pp.~1442--1449.

{{
\bibitem{Varanasi}
{\sc M. K. Varanasi}, {\em Decision feedback multiuser detection: A systematic
approach}, IEEE Trans. Inf. Theory, 45 (1999), pp.~219--240.
}}

\bibitem{Verdu1989}
{\sc S. Verd\'{u}}, {\em Computational complexity of optimum multiuser detection}, Algorithmica, 4 (1989), pp.~303--312.

\bibitem{Verdu}
{\sc S. Verd\'{u}}, {\em Multiuser Detection}, Cambridge Univ. Press, New York, 1998.

\bibitem{Waldspurger}
{\sc I. Waldspurger, A. Aspremont, and S. Mallat}, {\em Phase recovery, MaxCut and complex
  semidefinite programming}, Math. Program., 149 (2015), pp.~47--81.



{{
\bibitem{Xie}
{\sc Z. Xie, R. T. Short, and C. K. Rushforth}, {\em A family of suboptimum
detectors for coherent multi-user communications}, IEEE J. Sel. Areas Commun., 8 (1990), pp.~683--690.
}}

\bibitem{Yang}
{\sc S. Yang and L. Hanzo}, {\em Fifty years of MIMO detection: The road to large-scale MIMOs}, IEEE Commun. Surveys Tuts., 17 (2015), pp.~1941--1988.

\bibitem{Zhang}
{\sc S. Zhang and Y. Huang}, {\em Complex quadratic optimization and semidefinite programming}, SIAM J. Optim., 16 (2006), pp.~871--890.

\end{thebibliography}



\end{document}